\documentclass[12pt]{article}
\usepackage{bm, amsmath, amssymb, amsthm}

\newcommand{\R}{{\mathbb R}}
\newcommand{\Les}{{\mathbb L}}

\newcommand{\Sf}{{\mathbb S}}

\newtheorem{theorem}{Theorem}
\newtheorem{proposition}[theorem]{Proposition}
\newtheorem{lemma}[theorem]{Lemma}

\newtheorem{corollary}[theorem]{Corollary}

\newcommand{\spa}{\mbox{span}}
\newcommand{\Q}{\mathbb{Q}}
\newcommand{\Hy}{\mathbb{H}}
\def\<{\langle}
\def\>{\rangle}
\def\d{\partial}
\def\be{\begin{equation} }
\def\ee{\end{equation} }
\newcommand\bea{\begin{eqnarray*}}
\newcommand\eea{\end{eqnarray*}}
\begin{document}

\title{Minimal conformally flat hypersurfaces}

\author{C. do Rei Filho and R. Tojeiro\footnote{Partially supported by CNPq grant 	307338/2013-4  and FAPESP grant
2011/21362-2.}}
\date{}
\maketitle

\begin{abstract}
We study conformally flat hypersurfaces $f\colon M^{3} \to \Q^{4}(c)$ with three distinct principal curvatures and constant mean curvature $H$ in a space form with constant sectional curvature $c$. First we extend a theorem due to Defever when $c=0$ and show that there is no such   hypersurface  if $H\neq 0$. Our main results are for the minimal case $H=0$. If $c\neq 0$, we prove that if $f\colon M^{3} \to \Q^{4}(c)$ is a minimal conformally flat hypersurface with three distinct principal curvatures then $f(M^3)$ is an open subset of a generalized cone over a Clifford torus in an umbilical hypersurface $\Q^{3}(\tilde c)\subset \Q^4(c)$, $\tilde c>0$, with $\tilde c\geq c$ if $c>0$. For $c=0$, we show that, besides the   cone over the   Clifford torus in $\Sf^3\subset \R^4$, there exists precisely a one-parameter family of (congruence classes of) minimal isometric immersions $f\colon M^3 \to \R^4$  with three distinct principal curvatures of simply-connected conformally flat Riemannian manifolds. 
\end{abstract}

It was shown by E. Cartan \cite{ca} that if $f\colon M^{n} \to \Q^{n+1}(c)$  is a hypersurface of   dimension   $n\geq 4$ of a space form with constant sectional curvature $c$ and dimension $n+1$, then $M^n$ is conformally flat
 if and only if  $f$  has a principal curvature of multiplicity 
at least $n-1$. Recall that a Riemannian manifold $M^n$ 
is \emph{conformally flat} if each point of $M^n$ has an open neighborhood that is conformally diffeomorphic to 
an open subset of Euclidean space $\R^n$. 

Cartan also proved that any hypersurface $f\colon M^{3} \to \Q^{4}(c)$ with a principal curvature of multiplicity two is conformally flat, and  that the converse is no longer true in this case. 
The study of conformally flat hypersurfaces by Cartan was  taken up by  Hertrich-Jeromin \cite{h-j}, who showed  that a conformally flat hypersurface $f\colon\,M^3\to\Q^{4}(c)$ with three distinct principal curvatures  
 admits locally principal coordinates $(u_1, u_2, u_3)$ such that the induced metric $ds^2=\sum_{i=1}^3v_i^2du_i^2$ satisfies, say,
$v_2^2=v_1^2+v_3^2$. A recent improvement by  Canevari and the second author \cite{ct} (see Theorem \ref{main3} below) of  Hertrich-Jeromin's theorem is the starting point for the results of this paper.

Conformally flat hypersurfaces $f\colon M^{n} \to \Q^{n+1}(c)$ of dimension $n\geq 4$ with constant mean curvature 
were shown by do Carmo and Dajczer \cite{dcd} to be rotation hypersurfaces 
whose profile curves satisfy a certain ODE.  In particular, they concluded that minimal 
conformally flat hypersurfaces with  dimension $n\geq 4$ of $\Q^{n+1}(c)$ are generalized catenoids, extending a  
 previous result for $c=0$ by Blair \cite{bl}. The same conclusions apply for  
hypersurfaces $f\colon M^{3} \to \Q^{4}(c)$ that have a principal curvature of multiplicity two.  The case in which $f$ has three distinct principal curvatures was studied by Defever \cite{def}, who proved that no such hypersurface exists if $c=0$.  
Our first theorem extends this result  for $c\neq 0$.

\begin{theorem}\label{thm:cmc} There exists no  conformally flat hypersurface $f\colon\,M^3\to\Q^{4}(c)$ with three distinct principal curvatures and nonzero constant mean curvature.
\end{theorem}

Examples of minimal conformally flat hypersurfaces $f\colon M^{3} \to \Q^{4}(c)$ with three distinct principal curvatures can be constructed  for any $c\in \R$ as follows. Take a Clifford torus $g\colon \R^2\to \Q^{3}(\tilde c)$, $\tilde c>0$, with $\tilde c\geq c$ if $c>0$,   in an  umbilical hypersurface $\Q^{3}(\tilde c)\subset \Q^4(c)$, and consider the generalized cone over $g$ in $\Q^{4}(c)$, that is, 
the hypersurface parametrized, on the open subset of regular points, by 
the map $G\colon \R^2\times \R\to \Q^{4}({c})$ given by $$G(x, t)=\exp_{g(x)}(t\xi(g(x)))$$ where $\xi$ is a unit normal vector field to the inclusion 
$i\colon \Q^3(\tilde c)\to \Q^{4}({c})$ and  $\exp$ is the exponential map of  $\Q^{4}({c})$. 
If $c=0$, the map $G$ parametrizes a standard cone over a Clifford torus in the sphere $\Q^{3}(\tilde c)\subset \R^4$.\vspace{1ex}

The next result states that there are no further examples of minimal conformally flat hypersurfaces $f\colon M^{3} \to \Q^{4}(c)$ with three distinct principal curvatures if $c\neq 0$.

\begin{theorem}\label{thm:minimalcneq0} If $f\colon M^{3} \to \Q^{4}(c)$, $c\neq 0$,  is a minimal conformally flat hypersurface with three distinct principal curvatures then $f(M^3)$ is an open subset of a generalized cone over a Clifford torus in an umbilical hypersurface $\Q^{3}(\tilde c)\subset \Q^4(c)$, $\tilde c>0$, with $\tilde c\geq c$ if $c>0$.
\end{theorem}

 Our last and main result shows that the preceding statement is not true in  Euclidean space $\R^4$. In fact, we show that, besides  the  cone over a Clifford torus in $\Sf^3$,  there exists precisely a one-parameter family of further examples.

\begin{theorem}\label{thm:minimalceq0} There exists precisely a one-parameter family of (congruence classes of) minimal isometric immersions $f\colon M^3 \to \R^4$  with three distinct principal curvatures of simply connected conformally flat Riemannian manifolds. 
\end{theorem}

More precisely, we show that there exist an algebraic variety  $M^4\subset \R^6$, which contains a pair of straight lines $\ell_{-}$ and $\ell_+$ whose complement  $\tilde M^4=M^4\setminus (\ell_{-}\cup \ell_+)$ is a regular submanifold of $\R^6$,  
an involutive distribution ${\cal D}$ of codimension  one on $\tilde M^4$ and a finite group $G$ of involutions of $\tilde M^4$ isomorphic to $\mathbb{Z}_2\times \mathbb{Z}_2\times \mathbb{Z}_2\times \mathbb{Z}_2$ such that the following assertions hold:
\begin{itemize}
\item[(i)] To each leaf $\sigma$ of  ${\cal D}$ one can associate a minimal  immersion $f_\sigma\colon U_\sigma\to \R^4$
with three distinct principal curvatures of a simply connected open subset $U_\sigma\subset \R^3$, which is conformally flat with the metric induced by $f_\sigma$, and a covering map  $\phi_\sigma\colon U_\sigma\to \sigma$. The singular set $\ell_{-}\cup \ell_+$ of $M^4$ corresponds to the cone over a Clifford torus in $\Sf^3$.
\item[(ii)] If $\sigma$ and $\tilde \sigma$  are distinct leaves of  ${\cal D}$ then $f_{\tilde\sigma}$ is congruent to $f_\sigma$ 
if and only if there exist a diffeomorphism $\psi\colon U_\sigma\to U_{\tilde \sigma}$ and  $\Theta\in G$ such that $\phi_{\tilde \sigma}\circ \psi=\Theta\circ \phi_{\sigma}$. In particular, $\tilde \sigma=\Theta(\sigma)$.
\item[(iii)] If $f\colon M^3\to \R^4$ is a minimal isometric immersion  with three distinct principal curvatures of a simply connected
 conformally flat Riemannian manifold, then either $f(M^3)$ is an open subset of the cone over a Clifford torus in $\Sf^3$ or there exist a leaf $\sigma$ of ${\cal D}$ and a local diffeomorphism $\rho\colon M^3\to V$ onto an open subset $V\subset U_\sigma$ 
 such that $f$ is congruent to  $f_\sigma\circ \rho$.
\end{itemize}

\section[Preliminaries]{Preliminaries}

In this section we discuss a local  characterization of conformally flat hypersurfaces $f\colon\,M^3\to\Q^{4}(c)$ with three distinct principal curvatures 
 and present  the examples of minimal conformally flat hypersurfaces in $\Q^{4}(c)$ with three distinct principal curvatures given by generalized cones over Clifford tori. 

\subsection{Characterization of conformally flat hypersurfaces}

 First we  recall the notion of holonomic hypersurfaces.  One says that a  hypersurface $f\colon M^{n} \to \Q^{n+1}(c)$  is   \emph{holonomic} if $M^n$ carries  global orthogonal coordinates  $(u_1,\ldots, u_{n})$  such that the coordinate vector fields 
$\d_j=\dfrac{\partial}{\partial u_j}$ diagonalize the second fundamental form $I\!I$ of $f$. 

Set $v_j=\|\d_j\|$
 and define $V_{j} \in C^{\infty}(M)$, $1\leq j\leq n$, by 
 $I\!I(\d_j, \d_j)=V_jv_j$, $1\leq j\leq n$. 
Then the first and second fundamental forms of $f$ are
\be\label{fundforms}
I=\sum_{i=1}^nv_i^2du_i^2\,\,\,\,\mbox{and}\,\,\,\,\,I\!I=\sum_{i=1}^nV_iv_idu_i^2.
\ee
Denote $v=(v_1,\ldots, v_n)$ and  $V=(V_{1},\ldots, V_n)$. We call  $(v,V)$ the pair associated to $f$. 
The next result is well known.

\begin{proposition}\label{fund}
 The triple $(v,h,V)$, where 
 $h_{ij}=\frac{1}{v_i}\frac{\partial v_j}{\partial u_i},$
 satisfies the  system of PDE's
 
  \begin{eqnarray}\label{sistema-hol}
 \left\{\begin{array}{l}
  (i) \ \dfrac{\partial v_i}{\partial u_j}=h_{ji}v_j,\,\,\,\,\,\,\,\,\,(ii) \ \dfrac{\partial h_{ik}}{\partial u_j}=h_{ij}h_{jk},\vspace{1ex}\\
  (iii) \ \dfrac{\partial h_{ij}}{\partial u_i} + \dfrac{\partial h_{ji}}{\partial u_j} + \sum_{k\neq i,j} h_{ki}h_{kj} + 
   V_{i}V_{j}+cv_iv_j=0,\vspace{1ex}\\
  (iv) \ \dfrac{\partial V_{i}}{\partial u_j}=h_{ji}V_{j},\,\,\,\,1\leq i \neq j \neq k \neq i\leq n.
  \end{array}\right.
 \end{eqnarray}
Conversely, if $(v,h,V)$ is a solution of $(\ref{sistema-hol})$ on a simply connected open subset   $U \subset \R^{n}$, with  $v_i> 0$
everywhere  for all $1\leq i\leq n$,
then there exists a holonomic hypersurface $f\colon U \to  \Q^{n+1}(c)$  whose first and second fundamental forms are given by  $(\ref{fundforms}).$
\end{proposition}

The following characterization of conformally flat hypersurfaces $f\colon\,M^3\to\Q^{4}(c)$ with three distinct principal curvatures  was given in \cite{ct}, improving a theorem due to Hertrich-Jeromin \cite{h-j}. 

\begin{theorem}\label{main3} Let $f\colon\,M^3\to\Q^{4}(c)$ 
be a  holonomic hypersurface whose associated pair $(v, V)$ satisfies 
\be\label{holo.flat}\sum_{i=1}^3\delta_iv_i^2=0, \,\,\,\,\,\,\sum_{i=1}^3\delta_iv_iV_i=0\,\,\,\,\,\mbox{and}\,\,\,\,\,\sum_{i=1}^3\delta_iV_i^2=1, 
\ee
where $(\delta_1, \delta_2, \delta_3)= (1,-1, 1)$. Then $M^3$ is conformally flat and $f$ has three distinct principal curvatures. 

Conversely, any conformally flat hypersurface $f\colon\,M^3\to\Q^{4}(c)$ with three distinct 
principal curvatures is locally a 
holonomic hypersurface whose associated pair $(v, V)$ satisfies $(\ref{holo.flat}).$
\end{theorem}

It will be convenient to use the following equivalent version of Theorem~\ref{main3}.

\begin{corollary}\label{le:asspair}
Let $f\colon M^3\to \Q^4(c)$ be a holonomic hypersurface whose associated pair $(v,V)$ satisfies 
\begin{equation}\label{holo.flat.med.2}
\begin{array}{lcl}
v_2^2=v_1^2+v_3^2,  & \quad &
\displaystyle{V_2=-\frac{1}{3}\Big(\frac{v_1}{v_3}-\frac{v_3}{v_1}\Big) + \frac{v_2}{3}H}, \vspace{.2cm} \\
\displaystyle{V_1=-\frac{1}{3}\Big(\frac{v_2}{v_3}+\frac{v_3}{v_2}\Big) + \frac{v_1}{3}H}, & \quad &
\displaystyle{V_3=\frac{1}{3}\Big(\frac{v_1}{v_2}+\frac{v_2}{v_1}\Big) + \frac{v_3}{3}H},
\end{array}
\end{equation}
where   $H$ is
the mean curvature function of $f$. Then $M^3$ is conformally flat and $f$ has three distinct principal curvatures. 

Conversely, any conformally flat hypersurface $f\colon\,M^3\to\Q^{4}(c)$ with three distinct 
principal curvatures is locally a 
holonomic hypersurface whose associated pair $(v, V)$ satisfies $(\ref{holo.flat.med.2}).$
\end{corollary}
\begin{proof} It suffices to show that equations (\ref{holo.flat}) together with 
\begin{equation}\label{lem1.med}
\begin{array}{l}
H=\sum_{i=1}^3{V_iv_i^{-1}}
\end{array}
\end{equation}
are equivalent to (\ref{holo.flat.med.2}). For that, consider the Minkowski space $\Les^3$ endowed with the Lorentz inner product 
$$\<(x_1,x_2,x_3), (y_1,y_2,y_3)\> = x_1y_1-x_2y_2+x_3y_3.$$
Then the conditions in  (\ref{holo.flat}) say that  $v=(v_1, v_2, v_3)$ and $V=(V_1, V_2, V_3)$ are  orthogonal with respect to such inner product,  $v$ is light-like and $V$ is a unit space-like vector. Since $w=(-v_3,0,v_1)\in \Les^3$ is orthogonal to $v,$ we have  $v^\perp=\spa\{v,w\}.$ As $V\in v^{\perp},$ we can write $V=av+bw$ for some  $a,b \in C^{\infty}(M^3)$. Note that  $V_2=av_2.$ 
Using  $(\ref{holo.flat})$ we obtain
 $$1=\langle V,V \rangle = \langle av+bw,av+bw \rangle=b^2\langle w,w \rangle=b^2v_2^2.$$ 
 Thus $V=\frac{V_2}{v_2}v+\frac{\lambda}{v_2}w,$ with $\lambda=\pm 1.$ Therefore 
\begin{equation}\label{V1.e.V3.med.}
\begin{array}{lll}
\displaystyle{V_1=\frac{1}{v_2}(V_2v_1-\lambda v_3)}\qquad \mbox{and} \qquad \displaystyle{V_3=\frac{1}{v_2}(V_2v_3+\lambda v_1)}.
\end{array}
\end{equation}
Substituting  (\ref{V1.e.V3.med.}) in (\ref{lem1.med}) 
we obtain
\begin{equation}\label{V2.med.}
\begin{array}{l}
\displaystyle{V_2=-\frac{\lambda}{3}\Big(\frac{v_1}{v_3}-\frac{v_3}{v_1}\Big) + \frac{v_2}{3}H}.
\end{array}
\end{equation}
Substituting  (\ref{V2.med.}) in (\ref{V1.e.V3.med.}) yields 
$$V_1=-\frac{\lambda}{3}\Big(\frac{v_2}{v_3}+\frac{v_3}{v_2}\Big) + \frac{v_1}{3}H \,\,\,\,\,\,\mbox{and}\,\,\,\,\,
V_3=\frac{\lambda}{3}\Big(\frac{v_1}{v_2}+\frac{v_2}{v_1}\Big) + \frac{v_3}{3}H,$$
and changing the orientation, if necessary, we may assume that $\lambda=1$.\end{proof} 

\subsection{Generalized cones over Clifford tori}

First we show that, if  $g\colon \R^2 \to \Sf^3\subset \R^4$ is the Clifford torus parametrized by 
\be\label{eq:g}
  g(x_1,x_2)=\frac{1}{\sqrt{2}}(\cos (\sqrt{2} x_1), \sin (\sqrt{2} x_1),\cos (\sqrt{2} x_2), \sin (\sqrt{2} x_2)),
  \ee
   then the standard  cone
$F\colon (0, \infty)\times \R^2\to \R^4$ over $g$ given by
$$F(s,x)=sg(x),\,\,\,\,x=(x_1, x_2),$$
is a minimal conformally flat hypersurface.

The first and second fundamental forms of $F$  with respect to the unit normal vector field 
$$\eta(s, x_1, x_2)=\frac{1}{\sqrt{2}}(\cos (\sqrt{2} x_1), \sin (\sqrt{2} x_1),-\cos (\sqrt{2} x_2), -\sin (\sqrt{2} x_2))$$
are
$$
\begin{array}{lll}
I=ds^2+ s^2(dx_1^2+dx_2^2) \qquad \text{and} \qquad I\!I= s(-dx^2_1+dx^2_2).
\end{array}
$$
In terms of the new coordinates $u_1, u_2, u_3$, related to $s, x_1, x_2$ by
$$
u_2=\log s,\,\,\,\,u_1=\sqrt{2}x_1\,\,\,\,\mbox{and}\,\,\,\,\,u_3=\sqrt{2}x_2,
$$
the first and second fundamental forms of $F$ become
$$I=\frac{e^{2u_2}}{2}(du_1^2+2du_2^2+du_3^2)\,\,\,\,\,\mbox{and}\,\,\,\,\,I\!I=\frac{e^{u_2}}{2}(-du_1^2+du_3^2),
$$
hence $F$ is a minimal conformally flat hypersurface with three distinct principal curvatures, one of which being zero.\vspace{1ex}

The preceding example  can be extended to the case in which the ambient space is any space form, yielding  examples of minimal conformally flat hypersurfaces $f\colon M^3\to \Q^4(c)$ with three distinct principal curvatures  also for $c\neq 0$. \vspace{.5ex}

 Start with the Clifford torus $g\colon \R^2\to \Sf^3\subset \R^4$ parametrized by (\ref{eq:g}).
  If $c>0$, define $F\colon (0, \pi/\sqrt{c})\times \R^2\to \Sf^4(c)\subset \R^5=\R^4\times \R$ by
$$F(s,x)=\frac{1}{\sqrt{c}}(\cos ({\sqrt{c}}s) e_5+\sin ({\sqrt{c}}s) g(x)),$$
where $x=(x_1, x_2)$ and $e_5$ is a unit vector spanning the factor $\R$ in the orthogonal decomposition $\R^5=\R^4\times \R$.

Notice that, for each fixed $s=s_0$,  the map $F_{s_0}\colon \R^2\to \Sf^4(c)$, given by $F_{s_0}(x)=F(s_0, x)$, is also a Clifford torus 
in an umbilical hypersurface  $\Sf^3(\tilde c)\subset \Sf^4(c)$ with curvature $\tilde c=c/\sin^2(\sqrt{c}s_0)$, which has 
$$N_{s_0}=F_*\frac{\partial}{\partial s}|_{s=s_0}=-\sin(\sqrt{c}s_0)e_5+\cos(\sqrt{c}s_0)g$$
as a unit normal vector field along $F_{s_0}$. Notice also that
$$F(s+s_0, x)=\cos(\sqrt{c}s)F_{s_0}(x)+\sin(\sqrt{c}s)N_{s_0}(x),$$
thus $s\mapsto F(s,x)$ parametrizes the geodesic in $\Sf^4(c)$ through $F_{s_0}(x)$ tangent to $N_{s_0}(x)$ at $F_{s_0}(x)$.
Hence $F$ is a generalized cone over $F_{s_0}$.

The first and second fundamental forms of $F$  with respect to the unit normal vector field 
$$\eta(s, x_1, x_2)=\frac{1}{\sqrt{2}}(\cos (\sqrt{2} x_1), \sin (\sqrt{2} x_1),-\cos (\sqrt{2} x_2), -\sin (\sqrt{2} x_2), 0)$$
are
$$
\begin{array}{lll}
\displaystyle{I=ds^2+ \frac{1}{c}\sin^2 ({\sqrt{c}}s)(dx_1^2+dx_2^2)} \;\; \mbox{and} \;\; \displaystyle{I\!I=\frac{\sin ({\sqrt{c}}s)}{\sqrt{c}}(-dx^2_1+dx^2_2)}.
\end{array}
$$
In terms of the new coordinates $u_1, u_2, u_3$, related to $s, x_1, x_2$ by
\be\label{eq:uisxis}
\frac{du_2}{ds}=\frac{\sqrt{c}}{\sin(\sqrt{c}s)},\,\,\,\,u_1=\sqrt{2}x_1\,\,\,\,\mbox{and}\,\,\,\,\,u_3=\sqrt{2}x_2,\ee
the first and second fundamental forms of $F$ become
\be\label{eq:forms}
I=\frac{\sin^2 \theta}{2c}(du_1^2+2du_2^2+du_3^2)\,\,\,\,\,\mbox{and}\,\,\,\,\,I\!I=\frac{\sin \theta}{2\sqrt{c}}(-du_1^2+du_3^2),
\ee
where $\theta=\sqrt{c}s$, which, in view of the first equation in (\ref{eq:uisxis}), 
  satisfies 
$$\frac{d\theta}{du_2}=\sin \theta.$$
It follows from (\ref{eq:forms}) that $F$ is a minimal conformally flat hypersurface.

If $c<0$, define
$F\colon (0, \infty)\times \R^2\to \Hy^4(c)\subset \Les^5$ by
$$F(s,x)=\frac{1}{\sqrt{-c}}(\cosh ({\sqrt{-c}}s) e_5+\sinh ({\sqrt{-c}}s) g(x)),$$
where $x=(x_1, x_2)$,  $e_5$ is a unit time-like vector in $\Les^5$ and $e_5^\perp$ is identified with $\R^4$. 

As in the previous case, for each fixed $s=s_0$  the map $F_{s_0}\colon \R^2\to \Hy^4(c)$, given by $F_{s_0}(x)=F(s_0, x)$, is also a Clifford torus 
in an umbilical hypersurface  $\Sf^3(\tilde c)\subset \Hy^4(c)$ with curvature $\tilde c=-c/\sinh^2(\sqrt{-c}s_0)$, and 
$F$ is a generalized cone over $F_{s_0}$.

Now  the first and second fundamental forms of $F$ 
are
$$I=ds^2+ \frac{1}{-c}\sinh^2 ({\sqrt{-c}}s)(dx_1^2+dx_2^2)$$
and
$$I\!I=\frac{\sinh ({\sqrt{-c}}s)}{\sqrt{-c}}(-dx^2_1+dx^2_2).$$
In terms of the new coordinates $u_1, u_2, u_3$, related to $s, x_1, x_2$ by
$$
\frac{du_2}{ds}=\frac{\sqrt{-c}}{\sinh(\sqrt{-c}s)},\,\,\,\,u_1=\sqrt{2}x_1\,\,\,\,\mbox{and}\,\,\,\,\,u_3=\sqrt{2}x_2,
$$
they become
\be\label{eq:formsb}I=\frac{\sinh^2 \theta}{-2c}(du_1^2+2du_2^2+du_3^2)\,\,\,\,\,\mbox{and}\,\,\,\,\,I\!I=\frac{\sinh \theta}{2\sqrt{-c}}(-du_1^2+du_3^2),\ee
where $\theta(s)=\sqrt{-c}s$
satisfies 
$$\frac{d\theta}{du_2}=\sinh \theta.$$
It follows from  (\ref{eq:formsb}) that $F$ is a minimal conformally flat hypersurface with three distinct principal curvatures,
one of which being zero.

\section{The proofs of Theorems \ref{thm:cmc} and \ref{thm:minimalcneq0}}

First we derive a system of PDE's for new unknown functions associated to a conformally flat hypersurface $f\colon M^3\to \mathbb{Q}^4(c)$ with three distinct principal curvatures  under the assumption that 
$f$ has constant mean curvature.

\begin{proposition}\label{flat.med.reduz.result.}
Let $f\colon M^3\to \mathbb{Q}^4(c)$ be a holonomic hypersurface with constant mean curvature $H$ whose associated pair  $(v,V)$ satisfies $(\ref{holo.flat.med.2})$. Set $$(\alpha_1,\alpha_2,\alpha_3)=\Big(\frac{1}{v_2}\frac{\partial v_2}{\partial u_1},\frac{1}{v_3}\frac{\partial v_3}{\partial u_2},\frac{1}{v_1}\frac{\partial v_1}{\partial u_3}\Big).$$ 
Then $v_1, v_2, v_3, \alpha_1, \alpha_2, \alpha_3$ satisfy the differential equations  
\be\label{eq:e0}\frac{\partial v_1}{\partial u_1}=\frac{v_1}{v_2^4}(v_2^4+v_2^2v_3^2+v_3^4)\alpha_1,\,\,\,\,\,\frac{\partial v_2}{\partial u_1}=v_2\alpha_1,\,\,\,\,\,\frac{\partial v_3}{\partial u_1}=\frac{v_3^5}{v_2^4}\alpha_1,\ee
\be\label{eq:e0a}\frac{\partial v_1}{\partial u_2}=\frac{v_1^5}{v_3^4}\alpha_2,\,\,\,\,\,\frac{\partial v_2}{\partial u_2}=\frac{v_2}{v_3^4}(v_1^4-v_1^2v_3^2+v_3^4)\alpha_2,\,\,\,\,\,\frac{\partial v_3}{\partial u_2}=v_3\alpha_2,\ee
\be\label{eq:e0b}\frac{\partial v_1}{\partial u_3}= v_1\alpha_3,\,\,\,\,\,\frac{\partial v_2}{\partial u_3}=\frac{v_2^5}{v_1^4}\alpha_3,\,\,\,\,\,\frac{\partial v_3}{\partial u_3}=\frac{v_3}{v_1^4}(v_1^4+v_1^2v_2^2+v_2^4)\alpha_3,\ee
\be\label{eq:e1}\begin{array}{l}
{\displaystyle \frac{\partial \alpha_1}{\partial u_1}  =  \frac{1}{v_2^4}(3v_2^4-v_3^4)\alpha_1^2 - \frac{v_1^6}{v_3^8}(3v_1^2-2v_3^2)\alpha_2^2 + \frac{v_2^4}{v_1^2v_3^4}(3v_1^2+2v_2^2)\alpha_3^2 }  \vspace{.17cm} \\  
{\displaystyle +\ \frac{1}{9v_3^4}(5v_1^4+2v_2^2v_3^2)-  \frac{v_1^2v_2^2}{v_3^2}c +  \frac{v_1v_2}{18v_3^3}(v_2^2+v_3^2-2v_1v_2v_3H)H,} \end{array}
\ee
\be\label{eq:e1a}\frac{\partial \alpha_2}{\partial u_1}=2 \frac{v_1^2}{v_2^2}\alpha_1\alpha_2,\,\,\,\,\,\frac{\partial \alpha_3}{\partial u_1}=
2 \frac{v_3^2}{v_2^4}(4v_3^2+v_1^2)\alpha_1\alpha_3,\ee
\be\label{eq:e2}\begin{array}{l}
{\displaystyle \frac{\partial \alpha_2}{\partial u_2} = -\frac{v_3^4}{v_1^4v_2^2}(3v_2^2 + 2v_3^2) \alpha_1^2 -\frac{1}{v_3^4}(v_1^4 - 3v_3^4)\alpha_2^2 - \frac{v_2^6}{v_1^8}(2v_1^2+3v_2^2)\alpha_3^2 } \vspace{.17cm} \\ 
{\displaystyle -\frac{1}{9v_1^4}(5v_2^4-2v_1^2v_3^2)  +  \frac{v_2^2v_3^2}{v_1^2}c - \frac{v_2v_3}{18v_1^3}(v_1^2-v_3^2-2v_1v_2v_3H)H,} 
\end{array}
\ee
\be\label{eq:e2a}\frac{\partial \alpha_1}{\partial u_2}=2 \frac{v_1^2}{v_3^4}(4v_1^2- v_2^2)\alpha_1\alpha_2,\,\,\,\,\,\frac{\partial \alpha_3}{\partial u_2}=2 \frac{v_2^2}{v_3^2}\alpha_2\alpha_3,\ee
\be\label{eq:e2b}\frac{\partial \alpha_1}{\partial u_3}=-2 \frac{v_3^2}{v_1^2}\alpha_1\alpha_3,\,\,\,\,\,\frac{\partial \alpha_2}{\partial u_3}=2 \frac{v_2^2}{v_1^4}(4v_2^2- v_3^2)\alpha_2\alpha_3\ee
and
\be\label{eq:e3}\begin{array}{l}
{\displaystyle \frac{\partial \alpha_3}{\partial u_3} = \frac{v_3^6}{v_2^8}(2v_2^2+3v_3^2)\alpha_1^2 + \frac{v_1^4}{v_2^4v_3^2}(2v_1^2-3v_3^2)\alpha_2^2 + \frac{1}{v_1^4}(3v_1^4-v_2^4)\alpha_3^2 } \vspace{.17cm} \\ 
{\displaystyle + \frac{1}{9v_2^4}(5v_3^4+2v_1^2v_2^2) - \frac{v_1^2v_3^2}{v_2^2}c - \frac{v_1v_3}{18v_2^3}(v_1^2+v_2^2+2v_1v_2v_3H)H,} 
\end{array}
\ee
as well as the algebraic relations
\begin{equation}\label{equ.alg.f.e.}
\begin{array}{l}
\displaystyle{\Big(\frac{30v_1}{v_2v_3^9} F  - \frac{4v_1^4v_2^2}{v_3^6}(v_1^2-v_3^2)c + \frac{v_1^3v_2}{18v_3^7} m_2 H \Big)\alpha_2 =0}, \vspace{-.25cm} \\ \\
\displaystyle{\Big( \frac{30v_2}{v_1^9v_3}F  + \frac{4v_2^4v_3^2}{v_1^6}(v_1^2+v_2^2)c + \frac{v_3v_2^3}{18v_1^7} m_3 H \Big)\alpha_3=0} , \vspace{-.25cm} \\ \\
\displaystyle{\Big( \frac{30v_3}{v_1v_2^9} F  - \frac{4v_1^2v_3^4}{v_2^6}(v_2^2+v_3^2)c + \frac{v_1v_3^3}{18v_2^7}m_1 H \Big)\alpha_1=0 ,}
\end{array}
\end{equation}
 where
\begin{equation*}
\begin{array}{rcl}
m_1 &=& (v_2^2+4v_3^2)(4v_2^2+v_3^2)-8v_1v_2v_3(v_2^2+v_3^2)H,\vspace{.2cm} \\
m_2 &=& (v_1^2-4v_3^2)(4v_1^2-v_3^2)-8v_1v_2v_3(v_1^2-v_3^2)H, \vspace{.2cm}\\
m_3 &=& (v_1^2+4v_2^2)(4v_1^2+v_2^2)+8v_1v_2v_3(v_1^2+v_2^2)H, \vspace{.2cm}\\
F &=& \displaystyle{\frac{1}{27v_1v_2v_3}\big[9v_1^2v_3^8(v_2^2+v_3^2)\alpha_1^2  +   9v_1^8v_2^2(v_1^2-v_3^2)\alpha_2^2  - 9v_2^8v_3^2(v_1^2+v_2^2)\alpha_3^2} \vspace{.2cm}\\
 &&  - \ v_1^2v_2^2v_3^2(2v_1^2v_2^4-2v_1^2v_3^4-2v_2^2v_3^4-v_1^2v_2^2v_3^2)\big].
\end{array}
\end{equation*}
\end{proposition}
\proof
The triple $(v,h,V),$ where $h_{ij}=\frac{1}{v_i}\frac{\partial v_j}{\partial u_i},$ satisfies the system of PDE's
\begin{equation}\label{hol.flat.f.esp.}
\left\{\begin{array}{l}\displaystyle{
\!\!\!(i) \frac{\partial v_i}{\partial u_j}= h_{ji}v_j,  \,\,\,\,\,\,\, (ii) \frac{\partial h_{ij}}{\partial u_i} + \frac{\partial h_{ji}}{\partial u_j} +h_{ki}h_{kj}+ V_iV_j + cv_iv_j =0}, \vspace{.18cm}\\
\!\!\! \displaystyle{(iii) \frac{\partial h_{ik}}{\partial u_j}=h_{ij}h_{jk}, \,\,\,\,\,\,\, (iv) \frac{\partial V_i}{\partial u_j}=h_{ji}V_j, \hspace{.5cm} 1\leq i\neq j\neq k\neq i \leq 3,\vspace{.18cm}}\\
\!\!\! \displaystyle{(v) \delta_i\frac{\partial v_i}{\partial u_i}+ \delta_jh_{ij}v_j +\delta_kh_{ik}v_k=0, \, \ \ (vi) \delta_i\frac{\partial V_i}{\partial u_i}+ \delta_jh_{ij}V_j +\delta_kh_{ik}V_k=0.}
\end{array}\right. \!\!\!\!\!\!\!
\end{equation}
Equations  $(i)$, $(ii)$, $(iii)$ and $(iv)$ are due to the fact  that  $f$ is a holonomic hypersurface with $(v,V)$ as its associated pair, and  $(v)$ and $(vi)$ follow by differentiating (\ref{holo.flat}).

Using   (\ref{holo.flat.med.2}) and equations  $(i)$, $(iv)$, $(v)$ and $(vi)$ in  (\ref{hol.flat.f.esp.}) one can show that
\begin{equation}\label{rel.h-ki-kj}
 v_j^5h_{ki}=v_i^5h_{kj}, \quad 1\leq i\neq j\neq k\neq i\leq 3.
\end{equation}
From $(i)$ and $(v)$ of (\ref{hol.flat.f.esp.}), together with (\ref{rel.h-ki-kj}), one  obtains the formulae for  the derivatives $\frac{\partial v_i}{\partial u_j},$ $1\leq i,j\leq 3$. In a similar way, using  $(i)$, $(iii)$ and $(v)$, together with  (\ref{rel.h-ki-kj}), one finds the derivatives $\frac{\partial \alpha_i}{\partial u_j},$ $1\leq i\neq j\leq 3$.  In order to compute $\frac{\partial \alpha_i}{\partial u_i},$ $1\leq i\leq 3$, we note that equation  $(ii)$, together with (\ref{rel.h-ki-kj}) and the remaining equations in (\ref{hol.flat.f.esp.}), determines the   system of linear equations 
\begin{equation*}
MP=-B
\end{equation*}
in the variables  $\frac{\partial \alpha_i}{\partial u_i},$ $1\leq i\leq 3$, where
\begin{displaymath}
M= \left( \begin{array}{ccc}
 9v_1^2v_2^4v_3^2 & \frac{9v_1^8v_2^2}{v_3^2} & 0\vspace{.17cm}\\
 \frac{9v_1^2v_3^8}{v_2^2} & 0 & 9v_1^4v_2^2v_3^2 \vspace{.17cm}\\
 0 &  9v_1^2v_2^2v_3^4 & \frac{9v_2^8v_3^2}{v_1^2}
\end{array}\right)\!\!, \quad P= \left( \begin{array}{c}
\frac{\partial \alpha_1}{\partial u_1} \vspace{.17cm}\\
\frac{\partial \alpha_2}{\partial u_2} \vspace{.17cm}\\
\frac{\partial \alpha_3}{\partial u_3}
\end{array}\right)
\end{displaymath}
and
\begin{displaymath}
B= \left( \begin{array}{r}
- 9v_1^2v_3^4(v_2^2+v_3^2)\alpha_1^2 + \frac{9v_1^8v_2^2}{v_3^6}(4v_2^2+v_3^2)(v_1^2-v_3^2)\alpha_2^2 + 9v_2^8\alpha_3^2 \\- v_1^2v_2^2(2v_3^4-v_1^2v_2^2) + \ 9v_1^4v_2^4v_3^2 c -  v_1^3v_2^3v_3(v_1^2+v_2^2-v_1v_2v_3H)H. \vspace{0.32cm}\\
 - \frac{9v_1^2v_3^8}{v_2^6}(4v_1^2+v_2^2)(v_2^2+v_3^2)\alpha_1^2 + 9v_1^8\alpha_2^2 - 9v_2^4v_3^2(v_1^2+v_2^2)\alpha_3^2\\ - v_1^2v_3^2(2v_2^4+v_1^2v_3^2) + \ 9v_1^4v_2^2v_3^4 c +  v_1^3v_2v_3^3(v_1^2-v_3^2+v_1v_2v_3H)H. \vspace{0.32cm}\\
 9v_3^8\alpha_1^2 - 9v_1^4v_2^2(v_1^2-v_3^2)\alpha_2^2 + \frac{9v_2^8v_3^2}{v_1^6}(4v_3^2-v_1^2)(v_1^2+v_2^2)\alpha_3^2 \\ - v_2^2v_3^2(2v_1^4-v_2^2v_3^2) + \ 9v_1^2v_2^4v_3^4 c +  v_1v_2^3v_3^3(v_2^2+v_3^2+v_1v_2v_3H)H.
\end{array}\right)\!\!.
\end{displaymath}
One can check that  
such system has a unique solution given by (\ref{eq:e1}), (\ref{eq:e2}) and (\ref{eq:e3}).

Finally, computing the mixed derivatives  $\frac{\partial^2 \alpha_i}{\partial u_j\partial u_k} = \frac{\partial^2 \alpha_i}{\partial u_k\partial u_j},$ $1\leq i,k,j\leq 3$, from   (\ref{hol.flat.f.esp.}) we obtain 
$$0=\frac{\partial^2 \alpha_1}{\partial u_2\partial u_1} - \frac{\partial^2 \alpha_1}{\partial u_1\partial u_2} = \Big(\frac{30v_1}{v_2v_3^9} F  - \frac{4v_1^4v_2^2}{v_3^6}(v_1^2-v_3^2)c + \frac{v_1^3v_2}{18v_3^7} m_2 H \Big)\alpha_2, 
$$
$$
0=\frac{\partial^2 \alpha_2}{\partial u_3\partial u_2} - \frac{\partial^2 \alpha_2}{\partial u_2\partial u_3} = \Big( \frac{30v_2}{v_1^9v_3}F  + \frac{4v_2^4v_3^2}{v_1^6}(v_1^2+v_2^2)c + \frac{v_3v_2^3}{18v_1^7} m_3 H \Big)\alpha_3 
$$
and
$$
 0=\frac{\partial^2 \alpha_3}{\partial u_1\partial u_3} - \frac{\partial^2 \alpha_3}{\partial u_3\partial u_1} = \Big( \frac{30v_3}{v_1v_2^9} F  - \frac{4v_1^2v_3^4}{v_2^6}(v_2^2+v_3^2)c + \frac{v_1v_3^3}{18v_2^7}m_1 H \Big)\alpha_1.\qed
 $$

 In the  lemmata that follows we assume  the hypotheses of   Proposition \ref{flat.med.reduz.result.} to be satisfied and use the notations therein.

\begin{lemma} \label{le:v1equalv3} If $v_1=v_3$ everywhere then $H=0$.
\end{lemma}
\proof By the assumption and the first equation in (\ref{holo.flat.med.2}) we have $v_2=\sqrt{2}v_1$. We obtain from any two of the  equations in (\ref{eq:e0})  that $\alpha_1=0$, whereas any two of the  equations in (\ref{eq:e0b}) imply that $\alpha_3=0$.
Then (\ref{eq:e1}) and (\ref{eq:e3}) give
$$
\begin{array}{l}
18 \alpha_2^2 + 4 v_1^2 H^2 - 3 \sqrt{2}v_1 H + 36 v_1^2 c -18=0, \vspace{0.17cm} \\
18 \alpha_2^2 + 4 v_1^2 H^2 + 3 \sqrt{2}v_1 H + 36 v_1^2 c -18=0,
\end{array}
$$
which imply that $H=0$. \qed

\begin{lemma} \label{le:alphaiszero} The functions $\alpha_1, \alpha_2, \alpha_3$ can not vanish symultaneously on any open subset of $M^3$. 
\end{lemma}
\proof  If $\alpha_1, \alpha_2, \alpha_3$ all vanish on the open subset $U\subset M^3$, then (\ref{eq:e1}), (\ref{eq:e2}) and (\ref{eq:e3}) become 
\begin{equation}\label{cond.a1=a2=a3=0.1}
\begin{array}{l}
2(5v_1^4+2v_2^2v_3^2) - 18v_1^2v_2^2v_3^2c +  v_1v_2v_3(v_2^2+v_3^2-2v_1v_2v_3H)H=0,
\end{array}
\end{equation}
\begin{equation}\label{cond.a1=a2=a3=0.2}
\begin{array}{l}
2(5v_2^4-2v_1^2v_3^2) - 18v_1^2v_2^2v_3^2c + v_1v_2v_3(v_1^2-v_3^2-2v_1v_2v_3H)H=0,
\end{array}
\end{equation}
\begin{equation}\label{cond.a1=a2=a3=0.3}
\begin{array}{l}
2(5v_3^4+2v_1^2v_2^2) -  18v_1^2v_2^2v_3^2c - v_1v_2v_3(v_1^2+v_2^2+2v_1v_2v_3H)H=0. 
\end{array}
\end{equation}
Comparying (\ref{cond.a1=a2=a3=0.1}) with  (\ref{cond.a1=a2=a3=0.2}) and (\ref{cond.a1=a2=a3=0.3}) yields, respectively, 
\begin{equation}\label{cond.a1=a2=a3=0.3.1}
\begin{array}{ll}
\displaystyle{H = \frac{2(v_1^2+v_2^2)}{v_1v_2v_3} \qquad {\text{and}} \qquad H=-\frac{2(v_1^2-v_3^2)}{v_1v_2v_3}},
\end{array}
\end{equation}
which is a contradiction. \qed 

\begin{lemma} \label{le:alpha1alpha3} There does not exist any open subset of $M^3$ where  $v_1-v_3$ is nowhere vanishing and  $\alpha_1=0=\alpha_3$. 
\end{lemma}
\proof Assume that  $\alpha_1 = 0=\alpha_3$ and that  $v_1-v_3$ does not  vanish on the open subset $U \subset M^3$.
By Lemma \ref{le:alphaiszero},  $\alpha_2$ must be nonvanishing on an open dense subset $V\subset U$.  Then equations (\ref{eq:e0}), (\ref{eq:e0a}), (\ref{eq:e0b}), (\ref{eq:e1a}), (\ref{eq:e2a}) and (\ref{eq:e2b}) reduce to the following  on $V$:
\begin{equation*}
\begin{array}{l}
\displaystyle{\frac{\partial v_i}{\partial u_1}=\frac{\partial v_i}{\partial u_3}=\frac{\partial \alpha_i}{\partial u_j}=0}, \ \ i,j=1,2,3,  \ \  \ \ i\neq j, \hspace{2.08cm}
\end{array}
\end{equation*}
\begin{equation}\label{col.ex.f.e.1}
\begin{array}{l}
\displaystyle{\frac{\partial v_1}{\partial u_2}=\frac{v_1^5}{v_3^4}\alpha_2, \quad 
\frac{\partial v_2}{\partial u_2}=\frac{v_2}{v_3^4}(v_1^4-v_1^2v_3^2+v_3^4) \alpha_2, \quad
\frac{\partial v_3}{\partial u_2}=v_3\alpha_2},
\end{array}
\end{equation}
and, since $\alpha_1=\alpha_3=0,$ equations (\ref{eq:e1}), (\ref{eq:e2}) and (\ref{eq:e3})  become, respectively, 
\begin{equation}\label{col.ex.f.e.2}
\begin{array}{l}
\displaystyle{\frac{v_1^6}{v_3^8}(3v_1^2-2v_3^2)\alpha_2^2 + \frac{1}{9v_3^4}(5v_1^4+2v_2^2v_3^2) -  \frac{v_1^2v_2^2}{v_3^2}c}  \vspace{1ex}\\ \hspace*{20ex}\displaystyle{+\frac{v_1v_2}{18v_3^3}(v_2^2+v_3^2-2v_1v_2v_3H)H =0},
\end{array}
\end{equation}
\begin{equation}\label{col.ex.f.e.3}
\begin{array}{l}
\displaystyle{\frac{\partial \alpha_2}{\partial u_2} = -\frac{1}{v_3^4}(v_1^4 - 3v_3^4)\alpha_2^2 -\frac{1}{9v_1^4}(5v_2^4-2v_1^2v_3^2)  +  \frac{v_2^2v_3^2}{v_1^2}c} \vspace{1ex}\\ \hspace*{20ex}\displaystyle{- \frac{v_2v_3}{18v_1^3}(v_1^2-v_3^2-2v_1v_2v_3H)H},
\end{array}
\end{equation}
\begin{equation}\label{col.ex.f.e.4}
\begin{array}{l}
\displaystyle{\frac{v_1^4}{v_2^4v_3^2}(2v_1^2-3v_3^2)\alpha_2^2 + \frac{1}{9v_2^4}(5v_3^4+2v_1^2v_2^2) -  \frac{v_1^2v_3^2}{v_2^2}c} \vspace{1ex}\\ \hspace*{20ex} \displaystyle{-\frac{v_1v_3}{18v_2^3}(v_1^2+v_2^2+2v_1v_2v_3H)H=0}. 
\end{array}
\end{equation}
Multiplying (\ref{col.ex.f.e.2}) and   (\ref{col.ex.f.e.4}) by  $2v_3^8$ and $3v_1^2v_2^4v_3^2$, respectively,  and  subtracting one
from the  other, yield
\begin{equation}
\begin{array}{lcr}\label{col.ex.f.e.5}
\displaystyle{\alpha_2^2 = \frac{1}{90 v_1^6}\big[-2 v_1^2 v_2^2 v_3^2 (3 v_1^2 + 2 v_3^2)H^2 -  v_1 v_2 v_3  (6 v_1^4 + v_1^2 v_3^2 - 4 v_3^4)H} \vspace{0.15cm} \\  \hspace*{5ex} \displaystyle{- \ 18  v_1^2 v_2^2 v_3^2 (3 v_1^2 + 2 v_3^2)c + 2 (6 v_1^6 + 16 v_1^4 v_3^2 + 19 v_1^2 v_3^4 + 4 v_3^6)\big]}.
  \end{array}
\end{equation} 
On one hand, substituting (\ref{col.ex.f.e.5}) in  (\ref{col.ex.f.e.3}) we obtain
\begin{equation}\label{col.ex.f.e.6}
\begin{array}{l}
\displaystyle{ \frac{\partial \alpha_2}{\partial u_2} = \frac{1}{90 v_1^6 v_3^4}\big[2 v_1^2 v_2^2 v_3^2 (3 v_1^6 + 2 v_1^4 v_3^2 - 4 v_1^2 v_3^4 - 6 v_3^6)H^2} \vspace{0.1cm} \\\hspace*{9ex} \displaystyle{+ \  v_1 v_2 v_3 (6 v_1^8 + v_1^6 v_3^2 - 27 v_1^4 v_3^4 + 2 v_1^2 v_3^6 + 12 v_3^8)H} \vspace{0.1cm} \\\hspace*{9ex} \displaystyle{+ \ 18 v_1^2 v_2^2 v_3^2 (3 v_1^6 + 2 v_1^4 v_3^2 - 4 v_1^2 v_3^4 - 6 v_3^6)c} \vspace{0.1cm} \\\hspace*{9ex} \displaystyle{- \ 4 (v_1^4 - v_3^4) (3 v_1^6 + 8 v_1^4 v_3^2 + 16 v_1^2 v_3^4 + 6 v_3^6)  \big]}.
\end{array}
\end{equation}
On the other hand, differentiating  (\ref{col.ex.f.e.5}) with respect to  $u_2$ and using  (\ref{col.ex.f.e.1}) gives
\begin{equation}\label{col.ex.f.e.7}
\begin{array}{l}
\displaystyle{\alpha_2\frac{\partial \alpha_2}{\partial u_2} = -\frac{\alpha_2}{180 v_1^6 v_2 v_3^3} \big[v_2^2(-4 v_1^2 v_2 v_3^3 (5 v_1^4 - 4 v_1^2 v_3^2 - 6 v_3^4)H^2} \vspace{0.1cm} \\\hspace*{9ex}  \displaystyle{- \ v_1 v_3^2 (8 v_1^6 - 27 v_1^4 v_3^2 - 8 v_1^2 v_3^4 + 24 v_3^6)H} \vspace{1ex}\\\hspace*{9ex}\displaystyle{- \ 36 v_1^2 v_2 v_3^3 (5 v_1^4 - 4 v_1^2 v_3^2 - 6 v_3^4)c}  \vspace{1ex}\\\hspace*{9ex}\ \displaystyle{+8 v_2 v_3  (v_1^2 - v_3^2) (v_2^2 + v_3^2) (8 v_1^2 + 3 v_3^2)\big]}.
\end{array}
\end{equation}
Using that  $\alpha_2\neq 0$ and $v_1-v_3\neq 0$ on  $V$, we obtain from   (\ref{col.ex.f.e.6}) and (\ref{col.ex.f.e.7}) that
\begin{equation}\label{col.ex.f.e.8}
\begin{array}{l}
\displaystyle{ H^2 + \frac{4 v_1^6 - 2 v_1^4 v_3^2 - 9 v_1^2 v_3^4 + 4 v_3^6}{4 v_1^3 v_2 v_3 (v_1^2 - v_3^2)}H - \frac{2 v_2^2 (v_1^2 - v_3^2)}{v_1^4 v_3^2} + 9 c=0.}
\end{array}
\end{equation}
Differentiating (\ref{col.ex.f.e.8}) with respect to $u_2,$ and using  (\ref{col.ex.f.e.1}) we obtain
\begin{equation}\label{col.ex.f.e.9}
\begin{array}{l}
v_1 v_3 (22 v_1^6 - 11 v_1^4 v_3^2 - 4 v_1^2 v_3^4 + 8 v_3^6)H = -16  v_2^3 (v_1^2 - v_3^2)^2.
\end{array}
\end{equation}
Since $v_1-v_3\neq 0,$ we must have  $(22 v_1^6 - 11 v_1^4 v_3^2 - 4 v_1^2 v_3^4 + 8 v_3^6)\neq 0$ on $V.$ Therefore,  (\ref{col.ex.f.e.9}) implies   that
\begin{equation}\label{col.ex.f.e.10}
\begin{array}{l}
\displaystyle{H = -\frac{16  v_2^3 (v_1^2 - v_3^2)^2}{v_1 v_3 (22 v_1^6 - 11 v_1^4 v_3^2 - 4 v_1^2 v_3^4 + 8 v_3^6)}}.
\end{array}
\end{equation}
Finally, differentiating  (\ref{col.ex.f.e.10}) with respect to  $u_2$ and  using  (\ref{col.ex.f.e.1}) we obtain
\begin{equation}\label{col.ex.f.e.11}
\begin{array}{l}
\displaystyle{0 = -\frac{1680  v_1^5v_2^3 (v_1^2 - v_3^2)^2}{v_3 (22 v_1^6 - 11 v_1^4 v_3^2 - 4 v_1^2 v_3^4 + 8 v_3^6)^2}\alpha_2,}
\end{array}
\end{equation}
which is a contradiction, for the right-hand-side of  (\ref{col.ex.f.e.11}) is nonzero.\qed 

\begin{lemma} \label{le:alpha2alphaj} There does not exist any open subset of $M^3$ where  $\alpha_2=0=\alpha_j$ for some $j\in \{1, 3\}$.  
\end{lemma}
\proof We argue for the case in which $j=1$, the other case being similar. So, assume that  $\alpha_1$ and $\alpha_2$ vanish on an open subset  $U \subset M^3$.  By Lemma \ref{le:alphaiszero},  $\alpha_3$ is nonzero on an open dense subset $V\subset U$.  Equations (\ref{eq:e1}) and (\ref{eq:e2}) can be rewritten as follows on $V$: 

\begin{equation}\label{con.ad.f.e.5}
\begin{array}{l}
\displaystyle{\frac{v_2^4}{v_1^2v_3^4}(3v_1^2+2v_2^2)\alpha_3^2  + \frac{1}{9v_3^4}(5v_1^4+2v_2^2v_3^2)  -  \frac{v_1^2v_2^2}{v_3^2}c }\vspace{1ex}\\ \hspace*{20ex}\displaystyle{+  \frac{v_1v_2}{18v_3^3}(v_2^2+v_3^2-2v_1v_2v_3H)H =0}, \vspace{0.17cm} \\

\displaystyle{- \frac{v_2^6}{v_1^8}(2v_1^2+3v_2^2)\alpha_3^2 -\frac{1}{9v_1^4}(5v_2^4-2v_1^2v_3^2)  +  \frac{v_2^2v_3^2}{v_1^2}c} \vspace{1ex}\\ \hspace*{20ex}\displaystyle{- \frac{v_2v_3}{18v_1^3}(v_1^2-v_3^2-2v_1v_2v_3H)H =0}.
\end{array}
\end{equation}
Eliminating $\alpha_3^2$ from the equations in  (\ref{con.ad.f.e.5}) yields
\begin{equation}\label{con.ad.f.e.6}
\begin{array}{ll}
\displaystyle{H^2- \frac{7 v_1^4 + 7 v_1^2 v_3^2 + 2 v_3^4}{2 v_1 v_2 v_3 (v_1^2 + v_2^2)} H - \frac{ 2 v_3^2}{v_1^2 v_2^2} + 9 c =0}.
\end{array}
\end{equation}
Differentiating (\ref{con.ad.f.e.6}) with respect to $u_3$ and using that $\alpha_3\neq 0$ we obtain
\begin{equation}\label{con.ad.f.e.7}
\displaystyle{H=\frac{8 v_3^3  (v_1^2 + v_2^2)}{v_1  v_2 (21 v_1^4 + 21 v_1^2 v_3^2 + 4 v_3^4)}}.
\end{equation}
Finally, differentiating (\ref{con.ad.f.e.7}) with respect to  $u_3$ and using the fact that  $H$ is constant we obtain 
\begin{equation*}
\begin{array}{ll}
\displaystyle{0=\frac{120 v_3^3 v_2 (v_1^2 + v_2^2) (7 v_1^4 + 7 v_1^2 v_3^2 + 2 v_3^4)}{v_1^3 (21 v_1^4 + 21 v_1^2 v_3^2 + 4 v_3^4)^2},}
\end{array}
\end{equation*}
which is a contradiction. \qed

\begin{lemma} \label{le:alphaialphajneq0} If there exist $p\in M^3$ and $1\leq i\neq j\leq 3$  such that $\alpha_i(p)\neq 0\neq \alpha_j(p)$ then $H=0=c$.  
\end{lemma}
\proof  We give the proof for the case in which $i=1$ and $j=2$, the remaining ones being similar. Let $U\subset M^3$ be an open neighborhood of $p$ where  $\alpha_1$ and  $\alpha_2$ are nowhere vanishing. Then  (\ref{equ.alg.f.e.}) gives
$$\frac{30v_3}{v_1v_2^9} F  - \frac{4v_1^2v_3^4}{v_2^6}(v_2^2+v_3^2)c + \frac{v_1v_3^3}{18v_2^7}m_1 H=0   $$
and 
$$ \frac{30v_1}{v_2v_3^9} F  - \frac{4v_1^4v_2^2}{v_3^6}(v_1^2-v_3^2)c + \frac{v_1^3v_2}{18v_3^7} m_2 H =0,
$$
or equivalently,
\begin{equation}\label{con.ad.f.e.2}
\left\{\begin{array}{l}
F=-\frac{1}{540}v_1^2v_2^2v_3^2[Hm_1-72cv_1v_2v_3(v_2^2+v_3^2)], \vspace{0.18cm} \\
F=-\frac{1}{540}v_1^2v_2^2v_3^2[Hm_2-72cv_1v_2v_3(v_1^2-v_3^2)].
\end{array}\right.
\end{equation}
Subtracting one of the equations in  (\ref{con.ad.f.e.2}) from the other we obtain 
\begin{equation}\label{con.ad.f.e.3}
\begin{array}{lll}
\displaystyle{H^2-\frac{7(v_1^2 \ + \ v_2^2)}{8v_1v_2v_3}H+9c=0}.
\end{array}
\end{equation}
Differentiating  (\ref{con.ad.f.e.3}) with respect to  $u_1$ we obtain 
\begin{equation}\label{con.ad.f.e.4}
\begin{array}{lll}
\displaystyle{\frac{21v_3^3}{8v_1v_2^3}H\alpha_1=0}.
\end{array}
\end{equation}
Since  $\alpha_1\neq 0$,  equation  (\ref{con.ad.f.e.4}) implies that  $H=0$, and hence $c=0$ by (\ref{con.ad.f.e.3}). \vspace{1ex}\qed

\begin{lemma} \label{le:v1neqv3} If $v_1\neq v_3$ at some point of $M^3$ then $H=0=c$. 
\end{lemma}
\proof  Assume that $v_1(p_0)\neq v_3(p_0)$ for some  $p_0\in M^3$, and hence that $v_1\neq v_3$  on some open neighborhood  $U\subset M^3$
of $p_0$. By Lemma \ref{le:alphaiszero}, there exist an open subset $U'\subset U$ and $i\in\{1,2,3\}$ such that  $\alpha_i(p)\neq 0$ for all $p\in U'$.  It follows from Lemma \ref{le:alpha1alpha3} and Lemma \ref{le:alpha2alphaj} that there exist $q\in U'$ and $j\in\{1,2,3\}, \ j\neq i,$ such that   $\alpha_j(q)\neq 0$. Thus there exists $q\in M^3$ such that  $\alpha_i(q)\neq 0$ and $\alpha_j(q)\neq 0, \ i\neq j,$ and the conclusion follows from Lemma \ref{le:alphaialphajneq0}. \vspace{2ex}\qed

\noindent \emph{Proof of Theorem $\ref{thm:cmc}$:} Follows immediately from  Lemma \ref{le:v1equalv3} and Lemma \ref{le:v1neqv3}.\vspace{2ex}\qed

\noindent \emph{Proof of Theorem $\ref{thm:minimalcneq0}$:}  
Given $p\in M^3$, let  $u_1, u_2, u_3$ be $U$ be local principal coordinates on an open neighborhood  $U$ of $p$  as in 
Corollary \ref{le:asspair}. It follows from Lemma \ref{le:v1neqv3}  that the associated pair $(v,V)$ satisfies $v_1=v_3$ on $U$. Thus $\lambda_2$ vanishes on $U$, and hence everywhere on $M^3$ by analyticity. The statement is now a consequence of the next proposition. \qed

\begin{proposition}\label{thm:minimalpczero} Let $f\colon M^{3} \to \Q^{4}(c)$  be a  conformally flat hypersurface with three distinct principal curvatures. If one of the principal curvatures is everywhere zero,   then either $c=0$ and $f$ is locally a cylinder over a surface $g\colon M^2(\bar c)\to \R^3$ with constant Gauss curvature $\bar c\neq 0$ or $f$ is locally   a generalized cone over a surface $g\colon M^2(\bar c)\to \Q^{3}(\tilde c)$ with constant Gauss curvature $\bar c\neq \tilde c$ in an umbilical hypersurface $\Q^{3}(\tilde c)\subset \Q^4(c)$, $\tilde c\geq c$, with $\tilde c>0$ if $c=0$. If, in addition, $f$ is minimal, then $f(M^3)$ is an open subset of a generalized cone over a Clifford torus in an umbilical hypersurface $\Q^{3}(\tilde c)\subset \Q^4(c)$, $\tilde c>0$, with $\tilde c\geq c$ if $c>0$.
\end{proposition}
\proof Let $e_1, e_2, e_3$
denote local  unit vector fields which are principal directions correspondent to the distinct
principal curvatures $\lambda_1, \lambda_2, \lambda_3$, respectively. Then conformal flatness of $M^3$ is
equivalent to the relations
\be \label{eq:uno}
\<\nabla_{e_i}e_j,e_k\>=0
\ee
and
\be \label{eq:dos}
(\lambda_j-\lambda_k)e_i(\lambda_i)+(\lambda_i-\lambda_k)e_i(\lambda_j)+ (\lambda_j-\lambda_i)e_i(\lambda_k)=0,
\ee
for all distinct indices $i, j, k$ (see \cite{la}, p. 84). It follows from Codazzi's equation and (\ref{eq:uno}) that
\be\label{eq:tres}
\nabla_{e_i}e_i=\sum_{j\neq i}(\lambda_i-\lambda_j)^{-1}e_j(\lambda_i)e_j.
\ee
If, say, $\lambda_2=0$, then equation
(\ref{eq:dos}) yields
$$
\lambda_3^{-1}e_2(\lambda_3)=\lambda_1^{-1}e_2(\lambda_1):=\varphi,
$$
hence the distribution $\{e_2\}^\perp$ spanned by $e_1$ and $e_3$ is umbilical in $M^3$ by
(\ref{eq:tres}).

If $\varphi$ is identically zero on $M^3$, then $\{e_2\}^\perp$ is a totally geodesic distribution, and hence 
$M^3$ is locally isometric to a Riemannian product $I\times M^2$ by the local de Rham theorem. Since $M^3$ is conformally flat, it follows that $M^2$ must have constant Gauss curvature. Moreover, 
by Molzan's theorem (see Corollary $17$ in \cite{nol}), $f$ is locally an extrinsic product of isometric immersions of the factors, 
which is not possible if $c\neq 0$ because $f$ has three distinct principal curvatures. Therefore $c=0$ and $f$ is locally a cylinder over a surface with constant Gauss curvature in $\R^3$. 

If $\varphi$ is not identically zero on $M^3$, given $x\in M^3$ let $\sigma$ be the leaf of $\{e_2\}^\perp$ containing $x$ and let $j\colon\sigma\to M^3$ be the inclusion of  $\sigma$  into $M^3$. Denote  $\tilde g=f\circ j$. 
Then the normal bundle $N_{\tilde g}\sigma$ of $\tilde g$ splits as 
\begin{equation*}
N_{\tilde g}\sigma=f_*N_j\sigma\oplus N_fM=\spa\{f_*e_2\}\oplus N_fM
\end{equation*}
and
\bea
\tilde \nabla_X f_*e_2\!\!\!&=&\!\!\!f_*\nabla_Xe_2+\alpha^f(j_*X,e_2)\\
\!\!\!&=&\!\!\!-\varphi \tilde{g}_*X
\eea
for all $X\in \mathfrak{X}(\sigma)$, where $\tilde \nabla$ is the induced 
connection on ${\tilde g}^*T\Q^4(c)$. It follows that the normal vector field $\eta=f_*e_2$ of $\tilde g$ 
is parallel with respect 
to the normal connection of $\tilde g$, and that the shape operator of $\tilde g$ with 
respect to $\eta$  is 
given by 
$A^{\tilde g}_\eta=\varphi I$. 
It is a standard fact that this implies  $\tilde g(\sigma)$ to be contained in an umbilical hypersurface $\Q^{3}(\tilde c)\subset \Q^4(c)$, $\tilde c\geq c$, that is,  there exist an umbilical hypersurface $i\colon \Q^{3}(\tilde c)\to \Q^4(c)$   and an isometric immersion 
$g\colon M^2=\sigma\to\Q^3({\tilde c})$ such that $\tilde g=i\circ g$. 
Moreover,  since at any $y\in\sigma$ the fiber $L(y)\!=\!\spa\{\eta(y)\}$  
coincides with the normal space of $i$ at $g(y)$, it follows that $f$ coincides with 
the generalized cone over $g$ in a neighborhood of $x$. 

In particular, $M^3$ is a warped product $I\times_{\rho}M^2$, and since $M^3$ is conformally flat, $M^2$ must have constant Gauss curvature. If, in addition, $f$ is minimal, then $g$ must be a Clifford torus in an umbilical hypersurface $\Q^{3}(\tilde c)\subset \Q^4(c)$, $\tilde c>0$, with $\tilde c\geq c$ if $c>0$, and the preceding argument shows that $f(M^3)$ is an open subset of a generalized cone over~$g$.\qed

\section{Proof of Theorem \ref{thm:minimalceq0}}
First we rewrite Proposition (\ref{flat.med.reduz.result.}) when $H=0=c$ and state a converse to it.

\begin{proposition}\label{flat.min.reduz.result.}
Let $f\colon M^3\to \R^4$ be a holonomic hypersurface whose associated pair  $(v,V)$ satisfies 
\be\label{eq:vis}v_2^2=v_1^2+v_3^2\ee and 
\begin{equation}\label{holo.flat.min.2}
\begin{array}{l}
\displaystyle{V_1=-\frac{1}{3}\Big(\frac{v_2}{v_3}+\frac{v_3}{v_2}\Big)},\quad
\displaystyle{V_2=-\frac{1}{3}\Big(\frac{v_1}{v_3}-\frac{v_3}{v_1}\Big)}, \quad
\displaystyle{V_3=\frac{1}{3}\Big(\frac{v_1}{v_2}+\frac{v_2}{v_1}\Big)}.
\end{array}
\end{equation}
 Set $$\alpha=(\alpha_1,\alpha_2,\alpha_3)=\Big(\frac{1}{v_2}\frac{\partial v_2}{\partial u_1},\frac{1}{v_3}\frac{\partial v_3}{\partial u_2},\frac{1}{v_1}\frac{\partial v_1}{\partial u_3}\Big).$$
 Then  $\phi=(v_1,v_2,v_3,\alpha_1,\alpha_2,\alpha_3)$ satisfies the system of PDE's 
\begin{equation}\label{flat.min.reduz.}
\left\{\begin{array}{l}
\displaystyle{\frac{\partial \phi}{\partial u_1}=\Big(\frac{\partial v_1}{\partial u_1},v_2\alpha_1, \frac{v_3^5}{v_2^4}\alpha_1,\frac{\partial\alpha_1}{\partial u_1},2 \frac{v_1^2}{v_2^2}\alpha_1\alpha_2,2 \frac{v_3^2}{v_2^4}(4v_3^2+v_1^2)\alpha_1\alpha_3 \Big)}, \vspace{.25cm}\\
\displaystyle{\frac{\partial \phi}{\partial u_2}=\Big( \frac{v_1^5}{v_3^4}\alpha_2, \frac{\partial v_2}{\partial u_2}, v_3\alpha_2, 2 \frac{v_1^2}{v_3^4}(4v_1^2- v_2^2)\alpha_1\alpha_2,\frac{\partial\alpha_2}{\partial u_2},2 \frac{v_2^2}{v_3^2}\alpha_2\alpha_3 \Big)},\vspace{.25cm}\\
\displaystyle{\frac{\partial \phi}{\partial u_3}=\Big(   v_1\alpha_3,\frac{v_2^5}{v_1^4}\alpha_3, \frac{\partial v_3}{\partial u_3}, -2 \frac{v_3^2}{v_1^2}\alpha_1\alpha_3, 2 \frac{v_2^2}{v_1^4}(4v_2^2- v_3^2)\alpha_2\alpha_3,\frac{\partial\alpha_3}{\partial u_3} \Big),}
\end{array}\right.
\end{equation}
where
$$
\begin{array}{l}
\displaystyle \frac{\partial v_1}{\partial u_1}=\frac{v_1}{v_2^4}(v_2^4+v_2^2v_3^2+v_3^4)\alpha_1, \;\;\;
\displaystyle \frac{\partial v_2}{\partial u_2}=\frac{v_2}{v_3^4}(v_1^4-v_1^2v_3^2+v_3^4)\alpha_2, \vspace{.2cm}\\
\displaystyle \frac{\partial v_3}{\partial u_3}=\frac{v_3}{v_1^4}(v_1^4+v_1^2v_2^2+v_2^4)\alpha_3, \vspace{.2cm}\\
\displaystyle{\frac{\partial \alpha_1}{\partial u_1} = \frac{1}{v_2^4}(3v_2^4-v_3^4)\alpha_1^2 - \frac{v_1^6}{v_3^8}(3v_1^2-2v_3^2)\alpha_2^2 + \frac{v_2^4}{v_1^2v_3^4}(3v_1^2+2v_2^2)\alpha_3^2} \vspace{1ex}\\\hspace*{6ex} \displaystyle{+ \frac{1}{9v_3^4}(5v_1^4+2v_2^2v_3^2)},\vspace{.2cm}\\
\displaystyle{\frac{\partial \alpha_2}{\partial u_2} =-\frac{v_3^4}{v_1^4v_2^2}(3v_2^2 + 2v_3^2) \alpha_1^2 -\frac{1}{v_3^4}(v_1^4 - 3v_3^4)\alpha_2^2 - \frac{v_2^6}{v_1^8}(2v_1^2+3v_2^2)\alpha_3^2} \vspace{1ex}\\\hspace*{6ex}\displaystyle{-\frac{1}{9v_1^4}(5v_2^4-2v_1^2v_3^2)},\vspace{.2cm}\\
\displaystyle{\frac{\partial \alpha_3}{\partial u_3} = \frac{v_3^6}{v_2^8}(2v_2^2+3v_3^2)\alpha_1^2 + \frac{v_1^4}{v_2^4v_3^2}(2v_1^2-3v_3^2)\alpha_2^2 + \frac{1}{v_1^4}(3v_1^4-v_2^4)\alpha_3^2} \vspace{1ex}\\\hspace*{6ex}\displaystyle{+ \frac{1}{9v_2^4}(5v_3^4+2v_1^2v_2^2)},
\end{array}
$$
as well as  the algebraic equation 
\begin{equation}\label{equ.chave}
\begin{array}{rcr}
9v_1^2v_3^8(v_2^2+v_3^2)\alpha_1^2  +   9v_1^8v_2^2(v_1^2-v_3^2)\alpha_2^2  - 9v_2^8v_3^2(v_1^2+v_2^2)\alpha_3^2 \vspace{.2cm}\\
  - \ v_1^2v_2^2v_3^2(2v_1^2v_2^4-2v_1^2v_3^4-2v_2^2v_3^4-v_1^2v_2^2v_3^2) = 0.
\end{array}
\end{equation}
Conversely, if $\phi=(v_1,v_2,v_3,\alpha_1,\alpha_2,\alpha_3)$ is a solution of $(\ref{flat.min.reduz.})$ satisfying (\ref{eq:vis})  on an open simply-connected subset $U\subset \R^3$, then $\phi$ satisfies  (\ref{equ.chave}) and the  triple $(v, h, V)$, where $v=(v_1,v_2, v_3)$,  $V=(V_1, V_2, V_3)$ is given by $(\ref{holo.flat.min.2})$ and $h=(h_{ij})$, with  $h_{ij}=\frac{1}{v_i}\frac{\d v_j}{\d u_i}$, $1\leq i\neq j\leq 3$,  satisfies (\ref{sistema-hol}),  and hence gives rise to a holonomic  hypersurface $f\colon U\to \R^4$ whose associated pair $(v,V)$ satisfies (\ref{eq:vis}) and  (\ref{holo.flat.min.2}). 
\end{proposition}

In view of Corollary \ref{le:asspair} and Proposition \ref{flat.min.reduz.result.}, minimal conformally flat hypersurfaces of $\R^4$ are in correspondence with solutions $\phi=(v_1,v_2,v_3,\alpha_1,\alpha_2,\alpha_3)$  of (\ref{flat.min.reduz.}) satisfying (\ref{eq:vis}) and (\ref{equ.chave}). We shall prove  that such solutions  are, in turn,  in correspondence with the leaves of a foliation of codimension one on the algebraic variety constructed in the next result.

\begin{proposition}\label{prop.chave}  
Define $G,F:\R^6=\R^3\times \R^3\to\R$  by  
$$G(x,y)= x_2^2-x_1^2-x_3^2$$
 and
\begin{equation*}
\begin{array}{r}
F(x,y)= 9x_1^2x_3^8(x_2^2+x_3^2)y_1^2 + 9x_1^8x_2^2(x_1^2-x_3^2)y_2^2 - 9x_2^8x_3^2(x_1^2+x_2^2)y_3^2 \vspace{.2cm}\\
  - x_1^2x_2^2x_3^2(2x_1^2x_2^4-2x_1^2x_3^4-2x_2^2x_3^4-x_1^2x_2^2x_3^2).
\end{array}
\end{equation*}
Let
$M^4:=F^{-1}(0)\cap G^{-1}(0)\cap\{(x,y)\in \R^6;x_1>0,x_2>0,x_3>0\,\,\mbox{and}\,\,y\neq 0\}$
and let $\ell_{\pm}$ be the half lines in $M^4$ given by  $$\ell_{\pm}=\{(x,y)\in M^4\;:\; x=s(1, \sqrt{2},1)\,\,\mbox{for some $s>0$} \,\,\mbox{and}\,\, y=(0,\pm 1,0)\}.$$ Then $\tilde{M}^4=M^4\setminus (\ell_-\cup \ell_+)$ is a regular submanifold of $\R^6$ and $\ell_-\cup \ell_+$ is the singular set of $M^4$.
\end{proposition}
\proof If $p\in M^4$ we have $\nabla G(p)=\big(-2x_1,2x_2,-2x_3,0,0,0\big)$, while the components of $\nabla F(p)$ are given by 
$$x_1\frac{\partial F}{\partial x_1}(p)= 18x_1^8x_2^2(4x_1^2-3x_3^2)y_2^2 + 18x_2^{10}x_3^2y_3^2-2x_1^4x_2^2x_3^2(2x_2^4-x_2^2x_3^2-2x_3^4),$$
$$x_2\frac{\partial F}{\partial x_2}(p)=-18x_1^2x_3^{10}y_1^2-18x_2^8x_3^2(3x_1^2+4x_2^2)y_3^2 -2x_1^2x_2^4x_3^2(4x_1^2x_2^2-x_1^2x_3^2-2x_3^4),$$
$$x_3\frac{\partial F}{\partial x_3}(p)= 18x_1^2x_3^8 (3x_2^2+4x_3^2)y_1^2 - 18x_1^{10}x_2^2y_2^2 +2x_1^2x_2^2x_3^4(x_1^2x_2^2+4x_1^2x_3^2+4x_2^2x_3^2),$$
$$\frac{\partial F}{\partial y_1}(p)= 18x_1^2x_3^8(x_2^2+x_3^2)y_1,\,\,\,\,\,
\frac{\partial F}{\partial y_2}(p)=18x_1^8 x_2^2(x_1^2-x_3^2)y_2$$
and
$$\frac{\partial F}{\partial y_3}(p)=-18x_2^8(x_1^2+x_2^2)x_3^2y_3.$$
That  $M^4\setminus (\ell_+\cup \ell_-)$ is a smooth submanifold of $\R^6$ and $\ell_-\cup \ell_+$ is the singular set of $M^4$ is a consequence of the next two facts.\vspace{1ex}

\noindent {\bf Fact 1:}
$\nabla F(p)\neq 0$ for all $p\in M^4.$ 
\proof If $\nabla F(p)=0$ at $p=(x_1,x_2, x_3, y_1, y_2, y_3)$ then from $\frac{\partial F}{\partial y_1}(p)=0$ it follows that  $y_1=0$, whereas 
$\frac{\partial F}{\partial y_3}(p)=0$ implies that   $y_3=0$. Thus $y_2\neq 0$, and hence   $x_3=x_1$ from $\frac{\partial F}{\partial y_2}(p)=0$. Therefore $x_2=\sqrt{2}x_1,$ and then  $0=\frac{\partial F}{\partial x_2}(p)=-20\sqrt{2}x_1^{11},$ which  contradicts the fact that $x_1>0$. \vspace{1ex} \\
\noindent {\bf Fact 2:} The subset $\{p\in M^4\,:\,\nabla F(p)=a\nabla G(p)\,\,\,\mbox{for some}\,\,\,a\in \R-\{0\}\}$ coincides with $\ell_{-}\cup \ell_{+}.$ 
\proof Assume that 
\be\label{eq:nFnG}\nabla F(p)=a\nabla G(p)\ee for some $a\in \R-\{0\}$. Equation  (\ref{eq:nFnG})  gives us six equations, the last three of which yield $y_1=y_3=0$ and $x_3=x_1$. Since  $x_2^2=x_1^2+x_3^2,$ we obtain that $x_2=\sqrt{2}x_1.$ Using this and the second of such equations we obtain that  $a=-10x_1^{10}.$ Finally, the first one implies that $y_2^2=1$. \vspace{1ex} \qed

\begin{proposition}\label{prop.chaveb}  
Let  $X_1,X_2,X_3\colon M^4\to\R^6$ be defined by 
\begin{equation*}
\begin{array}{lcl}
\displaystyle{X_1(p)=\frac{1}{x_2^4}\Big((x_2^4+x_2^2x_3^2+x_3^4)x_1y_1,x_2^5y_1, x_3^5y_1,x_2^4 A_1(p)},\vspace{1ex}\\\hspace*{13ex} 2 x_1^2x_2^2y_1y_2,(8x_3^2+2x_1^2)x_3^2y_1y_3\Big), \vspace{.2cm} \\
\displaystyle{X_2(p)=\frac{1}{x_3^4}\Big( x_1^5y_2,(x_1^4-x_1^2x_3^2+x_3^4)x_2y_2,x_3^5y_2} ,\vspace{1ex}\\\hspace*{13ex}(8x_1^2- 2x_2^2)x_1^2y_1y_2, x_3^4A_2(p),2 x_2^2x_3^2y_2y_3 \Big), \vspace{.2cm} \\
\displaystyle{X_3(p)=\frac{1}{x_1^4}\Big( x_1^5y_3,x_2^5y_3, (x_1^4+x_1^2x_2^2+x_2^4)x_3y_3,-2x_1^2x_3^2y_1y_3},\vspace{1ex}\\\hspace*{13ex}(8x_2^2-2 x_3^2)x_2^2y_2y_3,x_1^4A_3(p) \Big),
\end{array}
\end{equation*}
 where 
\begin{equation*}
\begin{array}{l}
\displaystyle{A_1(p) = \frac{1}{x_2^4}(3x_2^4-x_3^4)y_1^2 - \frac{x_1^6}{x_3^8}(3x_1^2-2x_3^2)y_2^2 + \frac{x_2^4}{x_1^2x_3^4}(3x_1^2+2x_2^2)y_3^2} \vspace{1ex}\\\hspace*{10ex} \displaystyle{+ \frac{1}{9x_3^4}(5x_1^4+2x_2^2x_3^2)}, 
\end{array}
\end{equation*}
\begin{equation*}
\begin{array}{l}
\displaystyle{A_2(p) =-\frac{x_3^4}{x_1^4x_2^2}(3x_2^2 + 2x_3^2) y_1^2 -\frac{1}{x_3^4}(x_1^4 - 3x_3^4)y_2^2 - \frac{x_2^6}{x_1^8}(2x_1^2+3x_2^2)y_3^2} \vspace{1ex}\\\hspace*{10ex}\displaystyle{-\frac{1}{9x_1^4}(5x_2^4-2x_1^2x_3^2)},
\end{array}
\end{equation*}
\begin{equation*}
\begin{array}{l}
\displaystyle{A_3(p) = \frac{x_3^6}{x_2^8}(2x_2^2+3x_3^2)y_1^2 + \frac{x_1^4}{x_2^4x_3^2}(2x_1^2-3x_3^2)y_2^2 + \frac{1}{x_1^4}(3x_1^4-x_2^4)y_3^2} \vspace{1ex}\\\hspace*{10ex}\displaystyle{+ \frac{1}{9x_2^4}(5x_3^4+2x_1^2x_2^2)}.
\end{array}
\end{equation*}
Then the following assertions hold:
\begin{itemize}
\item[(i)]  $X_1(p),X_2(p),X_3(p)$ are linearly independent for all $p\in~\tilde M^4$.
\item[(ii)]   $\{p\in M^4;X_1(p)=0\}=\ell_-\cup \ell_+=\{p\in M^4;X_3(p)=0\}$.
\item[(iii)] The vector fields  $X_1,X_2,X_3$ are everywhere tangent to  $\tilde{M}^4$ and the curves 
 $\gamma_\pm\colon \R\to \R^6$ given by  $\gamma_\pm(t)=(e^t,\sqrt{2}e^t,e^t,0,\pm 1,0)$,
 are integral curves of $X_2$ with $\gamma_\pm(\R)=\ell_\pm$.
 \item[(iv)] $[X_i, X_j]=0$ on $\tilde{M}^4$ for all $1\leq i\neq j\leq 3$.
 \end{itemize}
 \end{proposition}
 \proof First notice that  $X_2(p)=0$ if and only if  $y_2=0$ and $A_2(p)=0$. Since $A_2(p)<0$ whenever $y_2=0$, it follows that
 $X_2(p)\neq 0$ for all $p~\in~M^4$. 

Now observe that  $X_1(p)=0$ if and only if $y_1=0$ and $A_1(p)=0$.  Thus, if  $p=(x_1, x_2, x_3, y_1, y_2, y_3)$ is such that $X_1(p)=0$ then  
\begin{equation*}
\left\{ \begin{array}{l}
ay_2^2+by_3^2 =c,\vspace{.2cm}\\
dy_2^2+ey_3^2 =f,\vspace{.2cm}\\
x_2^2=x_1^2+x_3^2,
\end{array}\right.
\end{equation*}
 where
$$a=9x_1^8x_2^4(3x_1^2-2x_3^2),\,\,\,\,b=-9x_2^8x_3^4(3x_1^2+2x_2^2),\,\,\,\, c= x_1^2x_2^4x_3^4(5x_1^4+2x_2^2x_3^2),$$
$$ d=9x_1^8x_2^2(x_1^2-x_3^2),\,\,\,\, e=-9x_2^8x_3^2(x_1^2+x_2^2)$$
and 
$$  f= x_1^2x_2^2x_3^2(2x_1^2x_2^4-2x_1^2x_3^2-2x_2^2x_3^4-x_1^2x_2^2x_3^2).$$
Such system has a unique solution given by
\begin{equation*}
y_2^2=\frac{x_3^8(x_1^2+x_2^2)^2}{9x_1^{12}} \quad \quad \mbox{and} \quad \quad y_3^2=-\frac{(x_1^2-x_3^2)^2}{9x_1^4}.
\end{equation*}
Thus we must have $y_3=0$, and hence  $x_1=x_3$ and $y_2=\pm 1.$  It follows that the subset  $\{p\in M^4;X_1(p)=0\}$  coincides with $\ell_-\cup \ell_+$.

In a similar way one shows that the subset  $\{p\in M^4;X_3(p)=0\}$  coincides with $\ell_-\cup \ell_+$, and the proof of $(ii)$ is completed.

To prove $(i)$, first notice that  $X_1(p),X_2(p),X_3(p)$ are  pairwise linearly independent. This already implies that if $\lambda_1,\lambda_2,\lambda_3\in \R$ are such that 
\begin{equation}\label{combina.linear}
\lambda_1X_1(p)+\lambda_2X_2(p)+\lambda_3X_3(p)=0
\end{equation}
then either $\lambda_1=\lambda_2=\lambda_3=0$ ou $\lambda_1\neq 0,\lambda_2\neq 0$ e $\lambda_3\neq 0.$ We will show that the last possibility can not occur. 

Equation (\ref{combina.linear}) gives the system of equations 
\begin{displaymath}
\left\{ \begin{array}{ll}
x_2^2x_3^4y_1\lambda_1+x_1^6y_2\lambda_2=0 \vspace{.2cm}\\
x_1^4x_3^2y_2\lambda_2+x_2^6y_3\lambda_3=0 \vspace{.2cm}\\
\displaystyle{\big[ A_1(p)-\frac{2}{x_2^4}(3x_2^4+x_3^4+2x_2^2x_3^2)y_1^2\big]\lambda_1=0}\vspace{.2cm} \\
\displaystyle{\big[A_2(p)-\frac{2}{x_3^4}(x_1^4+3x_3^4-2x_1^2x_3^2)y_2^2\big]\lambda_2=0} \vspace{.2cm}\\
\displaystyle{\big[ A_3(p)-\frac{2}{x_1^4}(3x_1^4+x_2^4+2x_1^2x_2^2)y_3^2 \big]\lambda_3=0.}
\end{array}\right.
\end{displaymath}
Thus, it suffices to prove that the system of equations
\begin{displaymath}
\left\{ \begin{array}{ll}
\displaystyle{ A_1(p)-\frac{2}{x_2^4}(3x_2^4+x_3^4+2x_2^2x_3^2)y_1^2=0}\vspace{.2cm} \\
\displaystyle{A_2(p)-\frac{2}{x_3^4}(x_1^4+3x_3^4-2x_1^2x_3^2)y_2^2=0 }\vspace{.2cm}\\
 \displaystyle{A_3(p)-\frac{2}{x_1^4}(3x_1^4+x_2^4+2x_1^2x_2^2)y_3^2 =0}
\end{array}\right.
\end{displaymath}
has no solutions for $p=(x_1, x_2, x_3, y_1, y_2, y_3)\in \tilde{M}^4$. We write the preceding system as a linear system 
\begin{displaymath}
\left\{ \begin{array}{rcl}
a_1y_1^2+a_2y_2^2+a_3y_3^2=a_4 \vspace{.2cm}\\
b_1y_1^2+b_2y_2^2+b_3y_3^2=b_4 \vspace{.2cm}\\
c_1y_1^2+c_2y_2^2+c_3y_3^2=c_4
\end{array}\right.
\end{displaymath}
in the variables $y_1^2,y_2^2$ and $y_3^2$, where
\begin{equation*}
\begin{array}{lll}
a_1= 9x_1^2x_3^8(3x_1^4+10x_2^2x_3^2), \hspace{.6cm} & b_1= 9x_1^4x_3^8(3x_2^2+2x_3^2), \vspace{.2cm}\\

a_2= 9x_1^8x_2^4(3x_1^2-2x_3^2),  & b_2= 9x_1^8x_2^2(3x_2^4-10x_1^2x_3^2), \vspace{.2cm}\\

a_3= -9x_2^8x_3^4(3x_1^2+2x_2^2), & b_3=9x_2^8x_3^4(2x_1^2+3x_2^2), \vspace{.2cm}\\

a_4= x_1^2x_2^4x_3^4(5x_1^4+2x_2^2x_3^2), & b_4= -x_1^4x_2^2x_3^4(5x_2^4-2x_1^2x_3^2), 
\end{array}
\end{equation*}
and
\begin{equation*}
\begin{array}{lll}
 c_1= 9x_1^4x_3^8(2x_2^2+3x_3^2),\,\,\,\,\,\,\,
c_2=9x_1^8x_2^4(2x_1^2-3x_3^2),\vspace{.2cm}\\

 c_3=-9x_2^8x_3^2(3x_3^4+10x_1^2x_2^2),\,\,\,\,\,\,\,
c_4= -x_1^4x_2^4x_3^2(5x_3^4+2x_1^2x_2^2).
\end{array}
\end{equation*}
Since
\begin{equation*}
\begin{array}{l}
\det
\left(\!\!\begin{array}{ccc}
a_1 & a_2 & a_3 \\
b_1 & b_2 & b_3 \\
c_1 & c_2 & c_3
\end{array}\!\right)=656100x_1^{12}x_2^{12}x_3^{12}(x_3^4+x_1^2x_2^2)(x_1^2-x_2^2+x_3^2)=0\vspace{.2cm}\\

\det
\left(\!\!\begin{array}{ccc}
a_1 & a_2 & a_4 \\
b_1 & b_2 & b_4 \\
c_1 & c_2 & c_4
\end{array}\!\right)=-29160x_1^{14}x_3^{18}x_2^6(x_3^4+x_1^2x_2^2)\neq 0,
\end{array}
\end{equation*}
such system has no solutions. Thus $(i)$ is proved.

Now, for $p\in M^4$ we have 
\begin{equation*}
\begin{array}{l}
\displaystyle{\langle \nabla F(p),X_1(p)\rangle=\frac{10y_1}{x_2^4}(x_2^4+x_3^4)F(p)=0=\langle \nabla G(p),X_1(p)\rangle}\vspace{.2cm} \\
\displaystyle{\langle \nabla F(p),X_2(p)\rangle=\frac{10y_2}{x_3^4}(x_1^4+x_3^4)F(p)=0=\langle \nabla G(p),X_2(p)\rangle}\vspace{.2cm} \\
\displaystyle{\langle \nabla F(p),X_3(p)\rangle=\frac{10y_3}{x_1^4}(x_1^4+x_2^4)F(p)=0= \langle \nabla G(p),X_3(p)\rangle},
\end{array}
\end{equation*}
 hence $X_1,X_2,X_3$ are everywhere tangent to  $\tilde{M}^4$. That  $\gamma_\pm$ is an integral curve of $X_2$  follows by 
checking that  $X_2(\gamma_\pm(t))=(e^t,\sqrt{2}e^t,e^t,0,0,0)=\gamma'_\pm(t)$.  Finally, a straightforward computation gives
\begin{equation*}
\begin{array}{l}
\displaystyle{[X_1,X_2](p)=\Big(0,0,0,\frac{10y_2}{9x_2^2x_3^{10}}F(p),\frac{10y_1}{9x_1^6x_2^4x_3^2}F(p),0\Big)=0} \vspace{.2cm} \\

\displaystyle{[X_1,X_3](p)=\Big(0,0,0,-\frac{10y_3}{9x_1^4x_2^2x_3^6}F(p),0,-\frac{10y_1}{9x_1^2x_2^{10}}F(p)\Big)=0} \vspace{.2cm} \\

\displaystyle{[X_2,X_3](p)=\Big(0,0,0,0,\frac{10y_3}{9x_1^{10}x_3^2}F(p),\frac{10y_2}{9x_1^2x_2^6x_3^4}F(p)\Big)=0}. \qed
\end{array} 
\end{equation*}

The proof of the next proposition is straightforward.

\begin{proposition}\label{prop:involutions} $(i)$  For each  $\epsilon=(\epsilon_1, \epsilon_2, \epsilon_3)$,  $\epsilon_j\in \{-1, 1\}$ for $1\leq j\leq 3$, the map $\Phi^{\epsilon}\colon \tilde{M}^4\to \tilde{M}^4$  given by 
$$\Phi^{\epsilon}(x_1, x_2, x_3, y_1, y_2, y_3)=(x_1, x_2, x_3, \epsilon_1y_1, \epsilon_2y_2, \epsilon_3y_3)$$
satisfies $\Phi^{\epsilon}_*X_j(p)=\epsilon_jX_j(\Phi^{\epsilon}(p))$ for all $p\in \tilde{M}^4$.\vspace{1ex}\\
$(ii)$ The map $\Psi\colon \tilde{M}^4\to \tilde{M}^4$  given by 
$$\Psi(x_1, x_2, x_3, y_1, y_2, y_3)=\bigg(x_3, x_2, x_1, \frac{x_2^4}{x_1^4}y_3, \frac{x_1^4}{x_3^4}y_2,\frac{x_3^4}{x_2^4}y_1\bigg)$$
satisfies 
$$\Psi_*X_1(p)=X_3(\Psi(p)),\;\;\Psi_*X_2(p)=X_2(\Psi(p))\;\;\mbox{and}\;\;\Psi_*X_3(p)=X_1(\Psi(p))$$
 for all $p=(x_1, x_2, x_3, y_1, y_2, y_3)\in \tilde{M}^4$.\vspace{1ex}\\
$(iii)$ The maps $\Psi$ and  $\Phi^{\epsilon}$, $\epsilon\in \{-1, 1\}\times \{-1, 1\}\times\{-1, 1\}$, generate a group of involutions
of $\tilde{M}^4$ isomorphic to $\mathbb{Z}_2\times \mathbb{Z}_2\times \mathbb{Z}_2\times \mathbb{Z}_2$ that preserves the distribution ${\cal D}$ spanned by the vector fields $X_1, X_2$ and $X_3$.
\end{proposition}

\noindent \emph{Proof of Theorem \ref{thm:minimalceq0}:} 
First we associate to each leaf $\sigma$  of ${\cal D}$ a  covering map $\phi_\sigma\colon U_\sigma\to \sigma$ from a simply-connected open subset $U_\sigma\subset \R^3$ and a minimal immersion $f_\sigma\colon U_\sigma\to \R^4$ with three distinct principal curvatures whose induced metric is conformally flat. 

For any $q\in \tilde M^4$ and for $1\leq i\leq 3$ denote  by $\tau_q^i\colon J_q^i\to \tilde{M}^4$ the maximal integral curve of $X_i$ through $q$, that is, $0\in J_q^i$, $\tau^i_q(0)=q$,  $(\tau^i_q)'(t)=X_i(\tau^i_q(t))$ for all $t\in J_q^i$, and $J_q^i$ is maximal with these properties. Let $${\cal D}(X_i)=\{(t, q)\in \R\times \tilde{M}^4\,:\, t\in J^i_q\}$$ and let $\varphi^i\colon  {\cal D}(X_i)\to \tilde{M}^4$ be the flow of $X_i$, given by $\varphi^i(t, q)=\tau^i_q(t)$. For a fixed $p\in \sigma$ define $U_\sigma=U_\sigma(p)$ by
$$U_\sigma=\{(u_1, u_2, u_3)\,:\, u_1\in J^1_p, \,u_2\in J^2_{\varphi^1(u_1, p)}, \, u_3\in J^3_{\varphi^2(u_2,\varphi^1(u_1, p))}\}$$
and $\phi_\sigma=\phi^p_\sigma$ by
$$\phi_\sigma(u_1, u_2, u_3)=\varphi^3(u_3, \varphi^2(u_2,\varphi^1(u_1, p))).$$
Then  $0\in U_\sigma$,   $\phi_\sigma(0)=p$,  and for all $u\in U_\sigma$ we have
\begin{equation*}
\frac{\partial \phi_\sigma}{\partial u_i}(u)=X_i(\phi_\sigma(u)), \,\, \,\,\, 1\leq i\leq 3.
\end{equation*} 
  
We claim that $\phi_\sigma$ is a covering map onto $\sigma$. Given $x\in\sigma$, let $\tilde B_{2\epsilon}(0)$ be an open 
ball of radius $2\epsilon$ centered at the origin such that
$\phi_\sigma^x|_{\tilde B_{2\epsilon}(0)}$ is a diffeomorphism onto 
$B_{2\epsilon}(x)=\phi_\sigma^x(\tilde B_{2\epsilon}(0))$. Since
\be\label{eq:psix0}
\phi_\sigma^p(t+s)=\varphi^3(t_3,\varphi^2(t_2,\varphi^1(t_1, \phi_\sigma^p(s))))=\phi_\sigma^{\phi_\sigma^p(s)}(t)
\ee
whenever both sides are defined, where $t=(t_1,t_2,t_3)$ 
and $s=(s_1,s_2,s_3)$, if $x=\phi_\sigma^p(s)$, $s=(s_1,s_2, s_3)\in U_\sigma$, then for any 
$y=\phi_\sigma^x(t)\in B_{2\epsilon}(x)$, $t=(t_1,t_2,t_3)\in \tilde B_{2\epsilon}(0)$, we have
$$ y=\phi_\sigma^x(t)=\phi_\sigma^{\phi_\sigma^p(s)}(t)=\phi_\sigma^p(s+t).$$
This shows that  $B_{2\epsilon}(x)\subset \phi_\sigma^p(U_\sigma)$ if  $x\in \phi_\sigma^p(U_\sigma)$, hence $\phi_\sigma^p(U_\sigma)$ is open in $\sigma$. But since $y=\phi_\sigma^x(t)$ if and only if $x=\phi_\sigma^y(-t)$, as follows from (\ref{eq:psix0}), the same argument shows that
$x\in \phi_\sigma^p(U_\sigma)$ if $y\in \phi_\sigma^p(U_\sigma)$ for some $y\in B_{2\epsilon}(x)$, and hence $\sigma\setminus \phi_\sigma^p(U_\sigma)$ is also open. It follows that $\phi_\sigma^p$ is onto $\sigma$.

Now, for any $x\in \sigma$ write 
$$
(\phi_\sigma^p)^{-1}(x)=\cup_{\alpha\in A}\tilde x_\alpha,
$$
 and for each $\alpha\in A$ let $\tilde B_{2\epsilon}(\tilde x_\alpha)$ denote 
the open ball of radius $2\epsilon$ centered at $\tilde x_\alpha$. Define 
a map $\psi_\alpha\colon B_{2\epsilon}(x)\to\tilde B_{2\epsilon}(\tilde x_\alpha)$
by 
$$
\psi_\alpha(y)=\tilde x_\alpha+(\phi_\sigma^x)^{-1}(y).
$$
By (\ref{eq:psix0}) we have 
\bea
\phi_\sigma^p(\psi_\alpha(y))\!\!\!&=&\!\!\!\phi_\sigma^p(\tilde x_\alpha+(\phi_\sigma^x)^{-1}(y))\\
\!\!\!&=&\!\!\!\phi_\sigma^{\phi_\sigma^p(\tilde x_\alpha)}((\phi_\sigma^x)^{-1}(y))\\
\!\!\!&=&\!\!\!\phi_\sigma^x((\phi_\sigma^x)^{-1}(y))\\
\!\!\!&=&\!\!\!y
\eea
for all $y\in B_{2\epsilon}(x)$. Thus $\phi_\sigma^p$ is a diffeomorphism 
from $\tilde B_{2\epsilon}(\tilde x_\alpha)$ onto 
$B_{2\epsilon}(x)$ having $\psi_\alpha$ as its inverse. 
In particular, this implies that $\tilde B_\epsilon(\tilde x_\alpha)$ and 
$\tilde B_\epsilon(\tilde x_\beta)$ are disjoint if $\alpha$ and 
$\beta$ are distinct indices in $A$. Finally, it remains to check that 
if $\tilde y\in (\phi_\sigma^p)^{-1}(B_\epsilon(x))$ then 
$\tilde y\in \tilde B_\epsilon(\tilde x_\alpha)$ for some $\alpha\in A$. 
This follows from the fact that 
$$
\phi_\sigma^p(\tilde y-(\phi_\sigma^x)^{-1}(\phi_\sigma^p(\tilde y)))=\phi_\sigma^{\phi_\sigma^p(\tilde y)}
(-(\phi_\sigma^x)^{-1}(\phi_\sigma^p(\tilde y)))=x.
$$
For the last equality, observe from (\ref{eq:psix0})  that for all 
$x, y\in \sigma$ we have that $\phi_\sigma^x(t)=y$ if and only if $\phi_\sigma^y(-t)=x$.

Writing $\phi_\sigma=(v_1,v_2,v_3,\alpha_1,\alpha_2,\alpha_3)$, it follows that $\phi_\sigma$ satisfies  (\ref{flat.min.reduz.}), as well as (\ref{eq:vis}) and  (\ref{equ.chave}). Defining $h_{ij}=\frac{1}{v_i}\frac{\d v_j}{\d u_i}$, $1\leq i\neq j\leq 3$, and $V=(V_1, V_2, V_3)$ by (\ref{holo.flat.min.2}),  it follows from  Proposition~\ref{flat.min.reduz.result.} that the triple $(v, h, V)$, where $v=(v_1,v_2, v_3)$ and $h=(h_{ij})$,  satisfies (\ref{sistema-hol}), and hence gives rise to a 
 minimal conformally flat hypersurface $f_\sigma\colon U_\sigma\to \R^4$ with three distinct principal curvatures by Corollary \ref{le:asspair}. 
 
 Given two distinct leaves $\sigma$ and $\tilde \sigma$  of ${\cal D}$, the corresponding immersions   $f_\sigma$ and $f_{\tilde \sigma}$ are  congruent if and only if there exists a diffeomorphism $\psi\colon U_\sigma\to U_{\tilde \sigma}$ such that 
 $$\psi^*I_{\tilde\sigma}=I_\sigma\,\,\,\,\mbox{and}\,\,\,\,\psi^*I\!I_{\tilde\sigma}=I\!I_\sigma$$
 where $I_\sigma$ and $I\!I_{\sigma}$ are the first and second fundamental formulae of $f_\sigma$, respectively, and $I_{\tilde\sigma}$, $I\!I_{\tilde\sigma}$ are those of $f_{\tilde \sigma}$. A long but straightforward computation shows that, up to a translation,  either $\psi$ coincides with the map given by
 $$\psi_\epsilon(u_1,u_2,u_3)=(\epsilon_1u_1, \epsilon_2u_2, \epsilon_3u_3)$$ 
 for some $\epsilon=(\epsilon_1, \epsilon_2, \epsilon_3)$ with $\epsilon_i\in \{-1, 1\}$ for $1\leq i\leq 3$, or it is the composition of such a map with the map given by
 $$\theta(u_1, u_2, u_3)=(u_3, u_2, u_1).$$
 It is easy to check that this is the case if and only if  there exists $\Theta\in G$ such that  $\phi_{\tilde \sigma}\circ \psi=\Theta\circ \phi_{\sigma}$.

 Now let $f\colon M^3\to \R^4$ be a  minimal isometric immersion  with three distinct principal curvatures of a simply connected
 conformally flat Riemannian manifold. We shall prove that either $f(M^3)$ is an open subset of the cone over a Clifford torus in $\Sf^3$ or there exist a leaf $\sigma$ of ${\cal D}$ and  a local diffeomorphism $\rho\colon M^3\to V$ onto an open subset $V\subset U_\sigma$ 
 such that $f$ is congruent to  $f_\sigma\circ \rho$.
 
 First we associate to $f$ a map $\phi_f\colon M^3\to M^4\subset \R^6$ as follows.
 Fix a unit normal vector field $N$ along $f$ and denote by $\lambda_1<\lambda_2<\lambda_3$ the distinct principal curvatures of $f$ with respect to $N$. For each $1\leq j\leq 3$, the eigenspaces $E_{\lambda_j}=\ker (A-\lambda_j I)$ associated to $\lambda_j$, where $A$ is the shape operator with respect to $N$ and $I$ is the identity endomorphism, form a field of directions along $M^3$, and since $M^3$ is simply connected we can find a smooth global unit vector field $Y_j$ along $M^3$ such that $\spa\{Y_j\}=E_{\lambda_j}$. 
Let the functions $v_1, v_2, v_3$ be defined on $M^3$ by
\be\label{eq:globalvj}v_j=\sqrt{\frac{\delta_j}{(\lambda_j-\lambda_i)(\lambda_j-\lambda_k)}},\,\,\,\delta_j=\frac{(\lambda_j-\lambda_i)(\lambda_j-\lambda_k)}{|(\lambda_j-\lambda_i)(\lambda_j-\lambda_k)|}, \,\,\,i\neq j\neq k\neq i,\ee
and let $\alpha_1, \alpha_2,\alpha_3$ be given   by
$$\alpha_1=\frac{v_1}{v_2}Y_1(v_2), \,\,\,\,\alpha_2=\frac{v_2}{v_3}Y_2(v_3)\,\,\,\mbox{and}\,\,\,\alpha_3=\frac{v_3}{v_1}Y_3(v_1).$$ 
Define $\phi_f\colon M^3\to \R^6$ by $\phi_f=(v_1,v_2,v_3,\alpha_1,\alpha_2,\alpha_3)$. 

Now, it follows from  Theorem \ref{main3} that each point $p\in M^3$ has a connected open neighborhood $U\subset M^3$ endowed with principal coordinates $u_1, u_2, u_3$ such that the pair  $(v,V)$ associated to $f$  satisfies (\ref{eq:vis}) and 
(\ref{holo.flat.min.2}). Let $\phi_U\colon U\to \R^6$ be given by $\phi_U=(v_1,v_2,v_3,\alpha_1,\alpha_2,\alpha_3)$,
with $(\alpha_1,\alpha_2,\alpha_3)=\big(\frac{1}{v_2}\frac{\partial v_2}{\partial u_1},\frac{1}{v_3}\frac{\partial v_3}{\partial u_2},\frac{1}{v_1}\frac{\partial v_1}{\partial u_3}\big)$. It is easy to check that $\phi_f|_U=\Theta\circ \phi_U$ for some $\Theta\in G$. 
On the other hand, by  Proposition  \ref{flat.min.reduz.result.} we have that $\phi_U$  
 satisfies (\ref{flat.min.reduz.}), as well as the algebraic equation 
(\ref{equ.chave}).  
It follows  that $\phi_U(U)\subset M^4$ and that 
$$\frac{\partial \phi_U}{\partial u_i}(u)=X_i(\phi_U(u)), \,\,\, \,\,\, 1\leq i\leq 3,$$
for all $u\in U$. 
Therefore either $\phi_U(U)$ is an open subset of a leaf $\sigma_{U}$ of the distribution ${\cal D}$ on $\tilde M^4$ spanned by $X_1, X_2, X_3$, or $\phi_U(U)$ is an open segment of either $\ell_+$ or $\ell_-$. If the latter possibility holds for some open subset $U\subset M^3$, then $v_1=v_3$ on $U$, hence $\lambda_2=0$ on $U$. By analyticity, $\lambda_2=0$ on $M^3$, and hence Proposition \ref{thm:minimalpczero} implies that  $f(M^3)$ is an open subset of a  cone over a Clifford torus in an umbilical hypersurface $\Q^{3}(\tilde c)\subset \R^4$, $\tilde c>0$.

Otherwise we have that each point $p\in M^3$ has an open neighborhood $U\subset M^3$ such that
$\phi_f(U)$ is an open subset of a leaf $\sigma_{U}$ of  ${\cal D}$. It follows that $\phi_f(M^3)$ is an open subset of a leaf $\sigma$ of  ${\cal D}$. If $\rho\colon M^3\to U_\sigma$ is a lift of $\phi_f$ with respect to $\phi_\sigma$, that is, $\phi_f=\phi_\sigma\circ \rho$, then $\rho$ is a local diffeomorphism such that $f$ and $f_\sigma\circ \rho$ have the same first and second fundamental forms.
Therefore $f$ is congruent to $f_\sigma\circ \rho$.\qed

{\renewcommand{\baselinestretch}{1}
\hspace*{-2ex}\begin{tabular}{ll}
Universidade Federal de Alagoas & Universidade Federal de S\~{a}o Carlos \\
Instituto de Matem\'atica & Departamento de Matem\'atica\\
57072-900 -- Macei\'o -- AL -- Brazil & 13565-905 -- S\~ao Carlos -- SP -- Brazil \\
E-mail: carlos.filho@im.ufal.br & E-mail: tojeiro@dm.ufscar.br
\end{tabular} }

\end{document}